\documentclass{amsart}
\usepackage{amsmath, amssymb}
\usepackage{amssymb}
\usepackage{tikz}
\def\a{\alpha}
\def\pa{\partial}
\def\de{\delta}
\def\De{\Delta}

\def\g{\gamma}

\def\G{\Gamma}
\def\phi{\varphi}

\def\Om{\Omega}

\def\si{\sigma}
\def\Si{\Sigma}

\def\R{{\mathbb R}}

\def\bS{{\bf S}}

\def\e{\varepsilon}
\def\wC{\widetilde C}
\def\wJ{\widetilde J}
\def\wU{\widetilde U}
\newtheorem{Th}{Theorem}[section]
\newtheorem{Lem}{Lemma}[section]
\newtheorem{Prop}{Proposition}[section]
\newtheorem{Def}{Definition}[section]
\newtheorem{Ex}{Example}[section]
\newtheorem{Pb}{Problem}[section]
\newtheorem{Rem}{Remark}
\numberwithin{equation}{section}
\usetikzlibrary{decorations.pathreplacing,decorations.markings}
\makeindex
\begin{document}

\title{The global study of Riemannian-Finsler geometry \\
\small{\it to the memory of Marcel Berger}}
\pagestyle{plain}
\author{Katsuhiro Shiohama}
\address {Institute of Information Science  \\
Fukuoka Institute of Technology  \\ 
Fukuoka, 811--0295, JAPAN}
\author{Bankteshwar Tiwari}
\address {Centre for Interdisciplinary Mathematical Sciences, Institute of Science   \\
Banaras Hindu University  \\
Varanasi 221 005, INDIA}
\thanks{First named author's work was supported by JSPS KAKENHI Grant Number15K04864}
\maketitle
\begin{abstract}
	The aim of this article is to present a comparative review of Riemannian and Finsler geometry. The structures of cut and	conjugate loci on Riemannian manifolds have been discussed by many	geometers including H. Busemann, M. Berger and W. Klingenberg. The key point in the study of Finsler manifolds is the non-symmetric property of its distance functions. We discuss fundamental results on the cut and conjugate loci of Finsler manifolds and note the differences  between Riemannian and Finsler manifolds in these respects. The topological and differential structures on Riemannian manifolds, in the presence of convex functions, has been an active field of research in the second half of $20^{\text{th}}$ century. We discuss some results on Riemannian manifolds with convex functions and their recently proved analogues in the field of Finsler manifolds.\\
	\indent	The final version of this article will appear in the book \textit{Geometry in History} (ed. S. G. Dani and A. Papadopoulos), Springer Verlag, 2019.
\end{abstract}
\bigskip



\tableofcontents
\section{Introduction.}
The origin of Finsler geometry can be traced back to Riemann's\index{Riemann, Bernhard, (1826--1866)} 1854 Habilitation address ``Uber die Hypothesen, welche der Geometrie zu grunde liegen" (On the Hypotheses which lie at the Foundations of Geometry), where he remarked: {\it `...The next case in simplicity includes those manifoldness in which the line-element may be expressed as the fourth root of a quartic differential expression. The investigation of this more general kind would require no really different principles, but would take considerable time and throw little new light on the theory of space, especially as the results cannot be geometrically expressed, I restrict myself, therefore, to those manifoldness in which the line-element is expressed as the square root of a quadratic differential expression...'}, translation by William Kingdon Clifford \cite{Riemann}. Later on, the geometry where the metric is the square root of a quadratic differential form,  got well recognized as Riemannian geometry. The general case was initiated by Paul Finsler in 1918 in his thesis written under the supervision of Carath\'eodory. It was said by S.S. Chern  that {\it Finsler Geometry is just Riemannian Geometry without the Quadratic Restriction} \cite{Chern}. In this article, we are interested in Global Finsler Geometry considered as an intrinsic metric geometry. We often refer to Riemannian geometry for our development of global Finsler geometry. One of the basic differences between  Riemannian and Finsler geometry is the possible asymmetry of distance functions. It turns out that in certain contexts Finsler geometry is more natural than Riemannian geometry, and closer to real world. Here is an example. On a slope of the earth's surface we may consider the ``distance" in terms of time taken to traverse it. Consider a person walking from the bottom of a hill to its top. In this context, the ``distance" will be larger from the bottom to the top, than from the top to the bottom. This example has been emphasized by Herbert Busemann,\index{Busemann, Herbert (1905--1994)} one of the most promieant promoters of Finsler geometry. Busemann's collected works were published in a $2$-volume set by Springer Verlag, see \cite{Busemann2}. Later Makoto Matsumoto explicitly showed that such metric is actually a Finsler metric, see \cite{Matsumoto}. 
\par  \medskip

Let us be more specific.

A Finsler metric on a smooth manifold is a smoothly varying family of Minkowski norms on the tangent spaces, rather than a family of inner products in the case of a Riemannian metric. It turns out that every Finsler metric induces an inner product, one in each direction of a tangent space at each point of the manifold. Thus, a Finsler metric associates to the manifold a family of inner products parametrized by the tangent spaces of the manifold (instead of being parametrized by the manifold, in the case of a Riemannian metric). However, the perpendicularity between two tangent vectors does not make sense on a Finsler manifold. Thus, it seems difficult to talk about the angle between two tangent vectors on such a manifold. In the mathematical literature, several kinds of connections were defined on a Finsler manifold. Some of the well-known connections were introduced  by J.L. Synge, J. H. Taylor, L. Berwald, E. Cartan, H. Rund, H. Hashiguchi and  S.S. Chern and others. In Riemannian geometry, the Levi-Civita connection is the canonical connection. It is torsion free and metrical. There is no  connection in Finsler geometry which is torsion free and metrical. There are different connections which have their own importance. The Chern connection\index{Chern connection} is important from two points of view: firstly when the Finsler metric induces a Riemannian metric, it reduces to the Levi-Civita connection, and secondly, it solves the problem of equivalence in Finsler geometry. This connection is torsion free but not metrical. \par
On Finsler manifolds, geometric objects are two-sided; viz., forward and backward, arising from the asymmetry of the distance function. The study of the cut locus and the conjugate locus of Riemannian and Finsler manifolds is important for the development of global Finsler geometry. In this article we give an overview of some aspects of global Riemannian geometry, developed in the very beginning of the last century, and of extensions of the Riemannian results on the cut locus and conjugate locus to Finsler manifolds. Among others, the cut locus is most important in the study of global Riemannian geometry. We discuss pointed  Blaschke-Finsler manifolds in connection with the Rauch conjecture on the cut locus and the conjugate locus of a compact simply connected Riemannian manifold. It should be emphasized that {\it convex sets and convex functions defined on a Finsler manifold are independent of the non-symmetric property of the distance function}. Hence, the notion of convexity is common to both Riemannian and Finsler geometries. 

The comparison theorems of Rauch, Berger and Toponogov play essential roles in the study of complete Riemannian manifolds of non-negative sectional curvature. However, we do not use these comparison theorems here in our study of Finsler manifolds. Following the ideas from Busemann\index{Busemann, Herbert (1905--1994)}~\cite{Busemann}, we discuss several topics on Finsler manifolds with non-symmetric distance functions. They are (i) the cut locus, (ii) the conjugate locus and (iii) convex sets including the Whitehead convexity theorem, (iv) convex functions, and (v) Busemann functions. We also discuss Busemann functions on both complete Riemannian and Finsler manifolds.  \par  \medskip
The article is organized as follows. Definitions and  notation are set up in \S 2. The forward cut locus and the forward conjugate locus and their fundamental properties, including the classical Whitehead convexity theorem are discussed in \S 3. A detailed discussion on cut locus and conjugate locus, including the classical results due to Klingenberg and Berger, which are very important in this article, are developed in \S 4. We discuss in \S 4, the well-known Blaschke problem on compact Finsler manifolds in connection with the Rauch conjecture~\cite{Rauch}. We discuss the simplest case of a pointed Blaschke manifold. Berger initiated the study of compact simply connected even-dimensional Riemannian manifolds of positive sectional curvature whose diameter is minimal~\cite{B2}, \cite{B1}. Omori \cite{Omori} discussed compact manifolds with minimal diameter with real analytic metric. In \S 5, we discuss the properties of Busemann functions and convex functions on complete non-compact Riemannian and Finsler manifolds. Finally, we summarize  Riemannian and Finsler results on convex functions. Some of these results have already been announced in \cite{SS} and \cite{ST}. For the basic tools in Riemannian and Finsler geometry we refer to \cite{Busemann}, \cite{Chern}, \cite{Klingenberg2}, \cite{CE},  \cite{Besse}, \cite{Sakai}, \cite{BCS}, \cite{CCL}.  \par \medskip
The authors would like to express their sincere thanks to Professor N. Innami,  Professor C. S. Aravinda and Professor Athanase Papadopoulos for reading and giving  their  valuable comments that improved this article. \par
\section{Definitions and preliminaries}
We first give the definitions of Riemannian and  Finsler metrics on a smooth manifold and discuss an important relation between them. The other notions that we present in this section are concerned with the non-symmetric properties of the distance function. 

\subsection{Riemannian and Finsler metrics.}\par
Let $M$ be a smooth manifold of dimension $n\ge 2$ and at each point 
$x\in M$, let $g_x$ be a dot product on the tangent space $T_xM$ to $M$. For smooth vector fields $X, Y$ defined in a neighborhood $U$ of $x$ in $M$, if the function $g(X,Y):U\to\R$ defined as $x\mapsto g_x(X(x),Y(x))$ is smooth, then $g$ is called a Riemannian metric,\index{Riemannian metric} and the pair $(M,g)$ is called a Riemannian\index{Riemannian manifold} manifold.\par
The tangent bundle $TM:=\cup_{x\in M}T_xM$ over $M$ is a smooth $2n$-manifold. Let $F:TM\to\R$ be a continuous function such that:
\begin{enumerate}
	\item $F$ is smooth on $TM\setminus\{0\}$ \ (regularity);
	\item $F(x,cu)=cF(x,u)$ for all $c>0$ and for all $(x,u)\in TM$ \ (positive homogeneity);
	\item $g_{ij}(x,u):=\frac{1}{2}\frac{\partial^2F^2(x,u)}{\partial u^i\partial u^j}$ is a positive definite matrix for all $(x,u)\in TM$ \ (strong convexity).
\end{enumerate} 
The pair $(M,F)$ is called a Finsler manifold\index{Finsler manifold} and $F$ its fundamental function\index{fundamental function}. The positive homogeneity and the strong convexity of $F$ leads us to the following facts:
\begin{Lem} [see \cite{BCS}]
	{\rm     
		Let $(M,F)$ be a Finsler manifold and $x\in M$. If $u,z,w\in T_xM$ and if $u$ is a non-zero vector, then we have
		\begin{enumerate}
			\item $g_{(x,u)}(z,w)=\frac{\partial^2F^2(x,u+sz+tw)}{2\partial t\partial s}|_{(0,0)}$;
			\item $g_{(x,u)}(u,u)=F^2(x,u)$;
			\item $g(x,tu)=g(x,u)$ for all $t>0$.
		\end{enumerate}          
	}     
\end{Lem}
\medskip
\subsection{Intrinsic distances and geodesics}
Let $(M,F)$ be a Finsler manifold of dimension $\ge 2$. For a smooth curve $c:[a,b]\to(M,F)$ the length 
$L(c)$ is given by
\[      L(c):=\int_a^b\,F(c(t), c'(t))\,dt, \ \ c'(t)=\frac{dc}{dt}.     \]
The reversed curve of  $c$, viz. $t\mapsto c(a+b-t)$, $t\in[a,b],$ is denoted by $c^{-1}$.  
The length of $c^{-1}$ is in general different from that of $c$:
\[      
L(c^{-1})=\int_a^b\,F(c^{-1}(t), (c^{-1})'(t))\,dt.\]     
The intrinsic distance $d(x,y)$ from a point $x\in M$ to a  point 
$y\in M$ is defined by
\[   d(x,y):=\inf\{L(c)\,|\, c\;\text{is a smooth curve from $x$ to $y$}\}.  \]            
We note that in gerneral $d(x,y)\neq d(y,x)$.
The indicatrix $\Si_x\subset T_xM$ at a point $x$ is the set of all unit vectors with respect to $F$:
\[       \Si_x:=\{u\in T_xM\,|\, F(x,u)=1    \}.       \]
The {\it reversibility constant\index{reversibility constant} $\lambda(C)$ of a compact set $C\subset M$} is defined by     
\begin{equation}\label{eq:reversibilityconstant}
\lambda(C):=\sup\left\{\frac{F(x,u)}{F(x,-u)}\,|\, x\in C,\quad u\in T_xM\setminus\{0\}\right\}.
\end{equation}
We then have 
\[     
\lambda(C)^{-1}d(x,y)\le d(y,x)\le \lambda(C)d(x,y),\ {\rm for \ all} \ x,y\in C.
\]        
Let $U$ be an open subset of a Finsler manifold $(M,F)$. Let $\nu U$ be the space of smooth vector fields on $U$ and $\nu U^+ \subset \nu U$ be the subset of nowhere vanishing vector fields. For $V \in \nu U^+$ and for all $X,Y \in \nu U$ define a  trilinear form $\langle \cdot \ , \cdot \ , \cdot \rangle_V$  by $\langle X,Y,Z \rangle_V=\frac{1}{4}\frac{\pa^3}{\pa r \pa s \pa t}F^2(V+rX+sY+tZ)|_{r=s=t=0}$, which is a symmetric $(0,3)$ tensor, called the Cartan tensor.\index{Cartan tensor} The Cartan tensor is a non-Riemannian quantity. It is easy to show that a Finsler metric reduces to a Riemannian metric if and only if its Cartan tensor vanishes.
An affine connection $\nabla^V$ is a map $\nabla^V:(X,Y)\in \nu U \times\nu U \to \nabla^V_XY\in \nu U $, linear in $Y$ (not necessarily linear in $X$) and satisfying the following conditions
$\nabla^V_X(fY)=f\nabla^V_XY+X(f) Y$ and $\nabla^V_{fX}Y=f\nabla^V_XY$ for all $f\in C^{\infty}U$ and $X,Y\in \nu U$. 
\begin{Th}{\rm [see Rademacher \cite{Rademacher}]} {\rm
		Let $(M,F)$ be a Finsler manifold, $U \subset M $ an open set and $V\in \nu U^+$, then there is a unique affine connection $\nabla^V$ associated with $V$, called the Chern connection,\index{Chern connection} satisfying the following conditions: 
		\begin{enumerate}
			\item $\nabla^V$ is torsion free, that is, $\nabla^V_XY-\nabla^V_YX=[X,Y] $ for all $X,Y \in \nu U$.
			\item $\nabla^V$ is almost metrical, that is, 
			\begin{equation*}
				Xg_V(Y,Z)=g_V(\nabla^V_XY,Z)+g_V(Y,\nabla^V_XZ)+2 \langle \nabla^V_XV,Y,Z\rangle_V \text{ for all} \ X,Y,Z \in \nu U .
			\end{equation*}
		\end{enumerate}
	}
\end{Th}

Using the connection $\nabla^V$, we introduce the covariant derivative $\frac{\nabla^V}{dt}$ along a smooth curve $c:[a,b]\to M$. For a vector field $X$ along the curve $c$ with tangent vector field $c'$, define $\frac{\nabla^V}{dt}X(t)=\nabla^V_{c'}X(t)$. If the vector fields $V$ and $c'$ along $c$ coincide, we also write $\frac{\nabla^V}{dt}X(t)=\frac{\nabla}{dt}X(t)$.\par 
Let $\g :[0,1]\to (M,F)$ be a   smooth curve on a Finsler manifold $(M,F)$. Then $\g$ is said to be a forward geodesic if $\frac{\nabla}{dt}\g '(t)=0$, for all $t \in [0,1]$.  In the local coordinates, if $\gamma'(t)=\frac{dx^i}{dt}\frac{\pa}{\pa x^i}$ and $\Gamma^k_{ij}(x,y)$ are components of the Chern connection (see \cite{Rademacher}, \cite{BCS}), then forward geodesics are the solutions of the second order non-linear differential equations 
\begin{equation*}\frac{d^2x^k}{dt^2}+\Gamma^k_{ij}(x,y)\frac{dx^i}{dt}\frac{dx^j}{dt}=0.
\end{equation*} A vector field $V\in \nu U$ is said to be a geodesic vector field if $\nabla^V_VV=0$, that is, if all the flow lines of $V$ are geodesics.  
\begin{Prop}[see Rademacher \cite{Rademacher}] \par
	{\rm
		Let $V$ be a nowhere-vanishing geodesic vector field defined on an open subset $U \subset M $. Denote by $\overline{\nabla}$, the Levi-Civita connection of the Riemannian manifold $(U,g_V)$ then   $\nabla^V_XV=\overline{\nabla}_XV$, for all vector fields $X$, in particular, the vector field $V$ is also a geodesic vector field for the Riemannian manifold $(U,g_V)$.
	}
\end{Prop} \medskip
\subsection{The exponential map and geodesic completeness} A forward geodesic $\g_u:[0,h)\to(M,F)$ with the initial conditions $\g_u(0):=x$ 
and $\dot\g_u(0):=u$ is the solution of a non-linear second order differential equation with smooth coefficients. Let $\Om_x\subset T_xM$ be the star-shaped domain with respect to the origin of $T_xM$, such that
\begin{equation}\label{eq:exponentialdomain}
\Om_x=\{ u\in T_xM\,|\, \g_u(1)\;\text{is defined}\}.    
\end{equation}      
We then define the exponential map $\exp_x:\Om_x\to(M,F)$ at $x$ by
\[       \exp_x\,u:=\g_u(1),\quad u\in \Om_x.  \]
We say that $(M,F)$ is {\it forward geodesically complete}  if $\Om_x=T_xM$  at some point $x\in M$. It then follows that $\Om_y=T_yM$ for all $y\in M$. The classical Hopf-Rinow theorem states that any two points on a forward geodesically complete $(M,F)$ are joined by a forward minimizing geodesic.\par
A forward (resp. backward) Cauchy sequence $\{q_j\}_{j=1}^{\infty}$ is defined by the condition that for every $\e>0$ there exists an integer $N_\e$ 
such that     
\[ 
d(q_j,q_k)<\e\, ({\rm resp.}\,  d(q_k,q_j)<\e), \  \text{for all $N_\e<j<k$}.   
\]
We say that $(M,F)$ is {\it forward complete {\rm (resp}. backward complete}) if every forward (resp. backward) Cauchy sequence converges.\par\medskip
\begin{Rem}\quad 
	{\rm
		If $F(p,u)=F(p,-u)$ holds for all $(p,u)\in TM$, then all the completeness conditions as above are equivalent to each other.}
	\end{Rem} 
	\begin{Prop} [see \cite{BCS}]
		{\rm
			If a Finsler manifold $(M,F)$ is forward geodesically complete, then $(M,F)$ is forward complete.
		}
	\end{Prop}

	\section{Forward cut Locus and forward conjugate locus}\par
	Let $(M,F)$ be a forward geodesically complete Finsler manifold, i.e. at each point $x\in M$ the exponential map $\exp_x: T_xM\to(M,F)$  is defined on the whole tangent space. 
	\subsection{Forward cut locus and forward conjugate locus}\index{forward cut locus}
	The forward cut locus\index{forward cut locus} and forward conjugate locus\index{forward conjugate locus} to a point $x \in (M,F)$, denoted by $C(x)$ and $J(x)$ respectively, are subsets of $M$ and have a significant role in the study of global differential geometry of Finsler manifolds. In particular, the forward cut locus to a point $x\in(M,F)$ {\it equipped with  the equivalence relation provided by the exponential map} contains all the topological information of $(M,F)$. We will define them shortly. Let $\g_u:[0,a]\to(M,F)$, for a unit vector $u\in\Si_x$, be a unit speed geodesic with $\g_u(0):=x$, $\dot\g_u(0)=u$. Define a function $i_x:\Si_x\to\R$ by
	\[       i_x(u):=\sup\{s>0\,|\, t=d(x,\g_u(t)), \ \text{for all} \ t\in(0,s)\}.   \] 
	The point $\g_u(i_x(u))$ is called the {\it forward cut point to $x$ along $\g_u$}, $\;u\in\Si_x$. In the case where $i(x)=\infty$, we call $x$ {\it a forward pole of $M$}. The forward injectivity radius function\index{injectivity radius function} at $x$ is defined by 
	\[ i(x):=\inf\{i_x(u)\,|\, u\in\Si_x\}. \]
	Let $v\in\Si_x$ be a unit vector with $g_u(u,v)=0$. Here we employ the Riemannian metric $g$ in (3.2) defined on $U_x\setminus\{x\}$ which will be described in the next section. We define the Jacobi field $Y_{u,v}:[0,a]\to TM$ along $\g_u$ such that $Y_{u,v}(0)=0$, $\nabla_{\frac{\partial}{\partial t}}Y_{u,v}(0)=v$. Namely, it is defined by
	\begin{equation}\label{eq:Jacobifield}  
	Y_{u,v}(t):=d(\exp_x)_{tu}tv,\quad v\in\Si_x,\quad g_u(u,v)=0,\quad t\in[0,a].
	\end{equation}
	The forward first conjugate point along $\g_u$ is defined as follows: Let $c_x:\Si_x\to(0,\infty)$ be a function defined by 
	\[     c_x(u):=\sup\{\,s>0\,|\, {\rm det}(d(\exp_x)|_{tu})\neq 0,\quad\text{for all} \ t\in(0,s)\},\quad u\in\Si_x.   \]
	In other words, a non-trivial Jacobi field $Y$ along $\g_u$, with $Y(0)=0$, exists such that $Y(c_x(u))=0$ and $Y(t)\neq 0$ for all $t\in(0,c_x(u))$. The point $\g_u(c_x(u))$ is called the {\it forward first conjugate point to $x$ along $\g_u$}.\par
	The forward tangential cut locus\index{forward tangential cut locus} $\wC(x)\subset T_xM$ and the forward first tangential conjugate locus\index{forward first tangential conjugate locus} $\wJ(x)\subset T_xM$  are defined by
	\[   \wC(x):=\{i_x(u)u\,|\, u\in\Si_x\},\quad \wJ(x):=\{c_x(u)u\,|\, u\in\Si_x\},     \] 
	and their exponential images are the forward cut locus and forward conjugate locus to $x$ respectively, and denoted by
	\[    C(x):=\exp_x\wC(x),\quad J(x):=\exp_x\wJ(x).     \]
	The domain containing the origin of $T_xM$ and bounded by $\wC(x)$ is denoted by $ \wU _x $. Clearly we have $ \partial \wU_x = \wC(x)$, and $\wU_x\subset T_xM $ is the maximal domain on which $ \exp_x $ is an embedding and denote $\exp_x \wU _x $ by $U_x$. 
	We observe from the definition of the cut locus to a point $x \in (M,F)$ that $C(x)$ contains all the information of $M$. In fact $ M \setminus C(x)$ is just an open disk and the identification structure of $\wC(x)$ via the exponential map defines the manifold.
	\par\medskip
	\subsection{Geodesic polar coordinates}\par
	We define  geodesic polar coordinates around an arbitrary fixed point $x\in M$. Let $\phi:\Si_x\times(0,i_x)\to(M,F)$ be defined by
	\[      \phi(u,t):=\exp_x\,tu,\quad 0<t<i_x.  
	\]
	The map $\phi$ is a diffeomorphism of $S^{n-1}\times (0,i_x)$ via identification of the indicatrix with the unit sphere $S^{n-1}\subset T_xM$ through the central projection. Property (3) in  Lemma 2.1 defines a Riemannian metric $g_u$ along $\g_u$. 
	Let $\xi$ be a radial vector field on $U_x$, i.e.,  $\xi(y):=d\text{exp}_{tu}(u), \, u\in\Sigma_x\ $, $y=\text{exp}_x(tu), \, 0<t<i_x$. Thus we have a smooth Riemannian metric $g$ on $U_x\setminus\{x\}$ defined by,
	\begin{equation}\label{eq:Riemannianmetric}
	g(y):=\bigcup_{y\in U_x}\,g(y,\xi(y)), \, y\in U_x\setminus\{x\} .          
	\end{equation}      
	The polar coordinates centered at $x$ are defined using $\phi$. All the $F$-geodesics emanating from $x$ are identified with geodesics as a Riemannian manifold $(U_x,g)$. The well-known first and second variation formulas  along a geodesic $\g_u$ with $u\in\Si_x$ are valid for $(U_x,g)$. Thus we know that $q:=\g_u(c_x(u))$ is a conjugate point to $x$ along $\g_u$ if and only if there is a non-trivial Jacobi field $Y$ along $\g_u$ such that $Y(0)=Y(c_x(u))=0$. If a unit speed geodesic $\si:[0,a]\to(M,F)$ admits a conjugate pair in its interior, then there is a $1$-parameter variation $\a:(-h,h)\times[0,a]\to M$ along $\si$ with $\a(\e,0)=x$ and $\a(\e,a)=\si(a)$ for all $\e\in(-h,h)$, such that all of its variational curves have lengths less than $a$. This means that
	\begin{equation}\label{eq:cutconjugatelocus}
	c_x(u)\ge i_x(u),\quad \text{for all} \ u\in\Si_x, \quad \text{for all} \  x\in M.
	\end{equation}
	We observe that if $\si:[0,a]\to(M,F)$ is a minimizing geodesic and if $\si(a)$ is conjugate to $\si(0)$ along $\si$, then $\si(a)$ is the cut point to $\si(0)$ along it.
	\par\medskip
	\subsection{The Whitehead Convexity Theorem}
	We define three kinds of convex sets on a complete Finsler manifold $(M,F)$.\par
	\begin{Def}
		{\rm
			A set $V\subset M$ is by definition {\it convex} if any pair of points $x,y\in V$ is joined by a unique minimizing geodesic whose image is contained entirely in $V$. The existence of a convex ball centered at every point on $(M,F)$ is stated in the Whitehead Convexity Theorem 3.1 below. Let $B(x,\de(x))$ be a convex $\de(x)$-ball around $x$. A closed set $V\subset M$ is called {\it locally convex} if every $x\in V$ has the property that $V\cap B(x,\de(x))$ is convex. A set $V\subset M$ is called {\it totally convex} if every geodesic joining two points in $V$ is contained entirely in $V$.
		}
	\end{Def}
	In the definition of locally convex sets above, the property of being closed is crucial; for, every open set would be locally convex. If two points $x$ and $y$ in a convex set $U$ are joined by a non-minimizing geodesic, then the latter is not necessarily contained in $U$.
	For example, a closed hemi-sphere in the standard sphere $\bS^n$ is locally convex and an open hemi-sphere is convex. $\bS^n$ itself is the only totally convex subset of itself. Every sublevel set $\phi^{-1}(-\infty,a]$ of a convex function $\phi:(M,F)\to\R$ defined on a complete Finsler manifold $(M,F)$ is totally convex.\par\medskip
	J. H. C. Whitehead investigated the injectivity radius in \cite{Whitehead2} and the convexity radius in \cite{Whitehead1}. We describe here some of his results: \par
	Let $\mathfrak U:=\cup_{x\in M}\wU_x\subset TM$. The natural projection $\Pi:TM\to M$ and the exponential map together define a smooth map $(\Pi,\exp):\mathfrak U\to M\times M$ by:
	\[ 
	(\Pi,\exp)(x,u):=(\Pi u,\,\exp_{_\Pi u}\, u)\in M\times M. 
	\]
	The image $(\Pi,\exp)(x,u)$ of $(x,u)\in\mathfrak U$ is the pair of  initial and end points of the geodesic $\g_u:[0,1]\to(M,F)$. Clearly, the zero section $O\subset\mathfrak U$ has the following property: $d(\Pi,\exp)|_O={\rm Identity}$; hence we have a small neighborhood $\Om\subset TM$ around the zero section and a small neighborhood $U(\De)\subset M\times M$ around the diagonal $\De$ of $M\times M$ such that
	\[     
	d(\Pi,\exp)|_\Om:\Om\to U(\De)\quad\text{is a diffeomorphism}.
	\]
	This fact means that any pair of points sufficiently close to each other is joined by a unique minimizing forward geodesic (compare \cite{Whitehead2}). If $C\subset M$ is a compact set, then there exists a  number $\a(C)>0$ such that if $x,y\in C$ satisfies $d(x,y)<\a(C)$, then
	there is a unique minimizing geodesic joining $x$ to $y$. Summing up, we have:
	\begin{Prop}
		{\rm
			Let $C\subset M$ be a compact set. Then for every point $x\in C$, its injectivity radius $i(x)$ is bounded below by a positive number $\a(C)$, i.e., $i(x)\ge\a(C)$ for all $x\in C$.
		}
	\end{Prop}   
	\begin{Th}[The Whitehead convexity theorem~\cite{Whitehead1}, ~\cite{Whitehead2}]\label{Th:Whitehead}\par
		{\rm
			\index{Whitehead convexity theorem}\index{Theorem!Whitehead convexity}There exists for every point $x\in(M,F)$, a positive number $\de(x)$ such that if $r\in(0,\de(x))$, then a forward metric $r$-ball $B(x,r):=\{y\in M\,|\, d(x,y)<r\}$ has the property that any pair of  points $y, z\in B(x,r)$ is joined by a unique minimizing geodesic whose image is contained entirely in the $B(x,r)$. 
		}
	\end{Th}
	\begin{proof}
		{\rm
			Let $C\subset M$ be a compact set and fix a small number $a>0$. Using the notations as in the last subsection, we consider a $4n-1$ dimensional smooth manifold 
			\[  \Lambda_{C,a}:=\{(u,v,t)\in\Si_x\times\Si_x\times(0,a],\,|\, x\in C,\quad g_u(u,v)=0,\quad t\in(0,a]\}.    \] 
			For an arbitrary fixed point $x\in C$, we shall employ the Riemannian metric $g$ in (\ref{eq:Riemannianmetric}). We then observe that the maps $\Lambda_{C,a}\to TM$:
			\[    (u,v,t)\to Y_{u,v}(t),\quad (u,v,t)\to\nabla_{\frac{\partial}{\partial t}}Y_{u,v}(t)   \]
			are smooth and uniformly bounded on $\Lambda_{C,a}$. We employ here the Riemannian connection $\nabla$ induced through the Riemannian metric $g$ in (\ref{eq:Riemannianmetric}). We then have from the construction of $Y_{u,v}$,
			\[   \frac{d}{dt}g_u(Y_{u,v},\nabla_{\frac{\partial}{\partial t}}Y_{u,v})(t)|_{t=0}=0,     \]    
			and 
			\[  \frac{d^2}{dt^2}g_u(Y_{u,v},\nabla_{\frac{\partial}{\partial t}}Y_{u,v})(t)|_{t=0}=1.  \]
			Then there exists a constant $0<\beta(C)\le 1$ independent of the choice of points on $C$, such that 
			\[  g_u(Y_{u,v},\nabla_{\frac{\partial}{\partial t}}Y_{u,v})(t)>0,\quad (u,v,t)\in\Lambda_{C,\beta(C)}.    \]
			
			\noindent Let $\lambda(C)>1$ be the reversibility constant for $C$. Let $B(x,r)\subset M$ be the forward metric $r$-ball around an arbitrary fixed point $x\in C$. If $y,z\in B(x,r)$, then the triangle inequality implies
			\[     d(y,z),\; d(z,y)<(1+\lambda(C))r.     \]
			Let $\si:[0,d(y,z)]\to (M,F)$ be a minimizing geodesic with $\si(0)=y$, $\si(d(y,z))=z$. We then observe that
			\[     d(x,\si(t))<(\lambda(C)+3)r/2,\quad\text{for all $t\in[0,d(y,z)]$ and for all $x\in C$}.  \]
			Therefore if $r<2\a(C)/(\lambda(C)+3)$, then every pair of points in $B(x,r)$ is joined by a unique minimizing geodesic whose image lies entirely in $B(x,r)$.\par
			We next take an arbitrary pair of points $y,z\in B(x,r)$ with $r\in(0,2\a(C))/(\lambda(C)+3)$. Let $\tau_t:[0,\ell_t]\to(M,F)$ be the unique minimizing geodesic with $\tau_t(0):=x$, $\tau_t(\ell_t):=\si(t)$, $t\in[0,d(y,z)]$. Here we set $\ell_t:=d(x,\si(t))$. The $1$-parameter family of geodesics $\{\tau_t\}_{t\in[0,d(y,z)]}$ form a geodesic variation, 
			along each $\tau_t$, $t\in[0,d(y,z)]$.  Clearly $t\mapsto\ell_t$ is a smooth function and the second variation formula implies that
			\[  \frac{d^2}{d\ell^2}\,\ell_t=g_u(Y_{u,v}(\ell_t),\nabla_{\frac{\partial}{\partial t}}Y_{u,v}(\ell_t))>0,\quad {\rm for \ all} \ (u,v,t)\in\Lambda_{C,r},   \]      
			if $r$ satisfies $r<\beta(C)$. We conclude the proof by setting
			\[    \de(x):=\min \left\{\frac{2\a(C)}{\lambda(C)+3},\;\frac{2\beta(C)}{\lambda(C)+3}\right\}.     \]
		}
	\end{proof}
	\begin{Rem}
		{\rm
			We define the convexity radius function $\de:(M,F)\to\R$ by:
			\[  \de(x):=\sup \{r>0\,|\, \text{every forward ball $B(y,t)$ contained in $B(x,r)$ is convex}\}.    \]  
		}
	\end{Rem}
	A relation between the injectivity radius and the convexity radius was first discussed by  Berger.\par\medskip
	\begin{Prop}[Berger's Remark -- oral communication]
		{\rm
			Let $(M,g)$ be a compact Riemannian manifold. Let $\de:(M,g)\to\R$ and $i:(M,g)\to\R$ be the convexity radius and the injectivity radius functions respectively. If $i(M)$ and $\de(M)$ be the infimum of $i$ and $\de$ over $M$ respectively,  then we have 
			\begin{equation}\label{eq:Berger0}
			\frac{1}{2}i(M)\ge\de(M).
			\end{equation}
		}
	\end{Prop}
	\begin{proof}
		{\rm
			Let $(M,g)$ be a compact Riemannian manifold. Suppose that (\ref{eq:Berger0}) does not hold. We then have $2\de(M)>i(M)$.     
			
			Choose points $x,y\in M$ such that $i(M)=d(x,y)=i(x)=i(y)$. Let $\g:[0,i(M)]\to M$ be a minimizing geodesic with $\g(0)=x$, $\g(i(M))=y$. Then,  Klingenberg's Lemma (see Proposition 4.3) implies that either $y$ is conjugate to $x$ along $\g$, or $\g$ extends to a simple closed geodesic such that $\g(0)=\g(2i(M))=x$. \par
			We first suppose that $\g:[0,2i(M)]\to M$ is a simple closed geodesic. Since $y\notin B(x,\de(M))$ and $x\notin B(y,\de(M))$, the midpoints $\g(i(M)/2)$ and $\g(3i(M)/2)$ between $x$ and $y$ are contained in $B(x,\de(M))$, which is  a contradiction. In fact, convexity of $B(x,\de(M))$ means that 
			$\g[i(M)/2,3i(M)/2]\subset B(x,\de(M))$.\par
			Suppose next that $y$ is conjugate to $x$ along $\g$. Let $\e$ be a number taken such that $\e\in(0,\frac{2\de(M)-i(M)}{2})$, and extend $\g$ to both sides and set $x':=\g(-\e)$, $q':=\g(i(M)+\e)$. Since $x', y'\in B(\g((i(M)/2),\de(M))$ and $\g|_{[-\e,i(M)+\e]}$ is not minimizing, we have a unique minimizing geodesic joining $x'$ to $y'$ whose image lies in $B(\g(i(M)/2),\de(M))$. Consequently, $\g[-\e,i(M)+\e]$ is not contained in $B(\g(i(M)/2),\de(M))$, a contradiction. This completes the proof.        
		}
	\end{proof}
	\begin{Pb}
		{\rm
			Is there any relation between the forward convexity radius and the forward injectivity radius on a compact Finsler manifold ?
		}
	\end{Pb}
	
	\par\medskip
	The non-symmetric property of the distance function on $(M,F)$ makes it difficult to address Problem 3.1. In particular, it is not clear whether we can find a simple closed geodesic $\g:[0,2i(M)]\to(M,F)$ if $i(\g(0))=i((M,F))$ and $q=\g(i(M,F))$ is not conjugate to $p$ along $\g$. This problem is discussed in \S 4.2.
	
	\begin{Ex}
		{\rm
			It should be remarked that Proposition 3.2 is optimal. In fact, for every rank-one symmetric space of compact type, equality holds in (\ref{eq:Berger0}).
			A simple example is given here. Let $T^2:=S^1(a)\times S^1(b)$ with $0<a<b$ be a flat torus whose fundamental domain is  rectangular with edge length $2a\pi$ and $2b\pi$, then $i(T^2)=a\pi$ and $\de(T^2)=\frac{1}{2}a\pi$. 
		}
	\end{Ex}
	\begin{Lem}
		{\rm
			The convexity radius function\index{convexity radius function} $\de:(M,F)\to\R$ is locally Lipschitz. 
		}
	\end{Lem}
	\begin{proof}
		{\rm
			Let $\lambda(C)$ be the reversibility constant of a compact set $C\subset M$ as defined in (\ref{eq:reversibilityconstant}). We take two arbitrary  points $x,y\in C$ sufficiently close to each other. We then observe that $B(y,\de(x)-d(x,y))\subset B(x,\de(x))$ implies that $\de(y)\ge\de(x)-d(x,y)$ and similarly,
			$B(x,\de(y)-d(y,x))\subset B(x,\de(x))$ implies $\de(x)\ge\de(y)-d(y,x)$. We conclude the proof by noting that
			\[   \frac{|\de(x)-\de(y)|}{d(x,y)},\;\frac{|\de(x)-\de(y)|}{d(y,x)}\le\lambda(C),\quad \text{for all} \ x,y \in C.   \]
		}\end{proof}
		\par\medskip
		Using this notion of  convexity radius function, Theorem \ref{Th:Whitehead} states that the forward distance function from a point on the manifold to $B(x,r)$ for $r\in(0,\de(x))$ is convex. 
		\par\medskip
		
		\section{The properties of cut locus and conjugate locus}\par
		From now on, for simplicity, we shall only discuss forward geodesics, forward cut locus, forward conjugate locus, etc.  The case where backward aspects are needed is rare. We shall state basic properties of cut locus and conjugate locus which have been discussed in Riemannian geometry by Whitehead, Myers, Klingenberg, Berger, Omori and many others. Rademacher has discussed in \cite{Rademacher} the Finsler version of some results in the proof of the classical sphere theorem. We summarize them here.
		\par\medskip
		\subsection{Foots of closed sets}
		Let $C\subset M$ be a closed set and let $x\in M\setminus C$. A point $y\in C$ is called a {\it foot of $x$ on $C$} if and only if $y$ satisfies $d(x,y)=d(x,C):=\inf\{d(x,y)\,|\,y\in C\}$. Let \[ B^{-1}(C,r):=\{x\in M\,|\, d(x,C)<r \} \ {\rm and} \ S^{-1}(C,r):=\{y\in M\,|\, d(y,C)=r \} \] be the backward metric $r$-ball and the backward metric $r$-sphere around a closed set $C\subset M$ respectively. Let $\overline{B}^{-1}(C,r)$ be the closure of $B^{-1}(C,r)$. If a point $x\in M\setminus C$ has more than one feet, then $x$ belongs to the backward cut locus to $C$.\par
		With this notation the following Lemma (4.1) shows the common property of a foot of a point on a closed set in both Riemannian and Finsler geometry. We show how the triangle inequalities are employed with a non-symmetric distance function.
		\begin{Lem}\label{Lem:foot}
			{\rm
				Let $C\subset M$ be a closed set. Take a point $x\in M \setminus C$ and a positive number $r$ such that $d(x,C)>r$. Let $y\in S^{-1}(C,r)$ be a foot of $x$ on $S^{-1}(C,r)$ and  $\g:[0,a]\to(M,F)$ is a unit speed minimizing geodesic with $\g(0)=x, \g(a)=y$. Then its extension reaches a point  on $C$ at $\g(a+r)$, which is the unique foot of $y$ on $C$. 
			}
		\end{Lem}
		\begin{proof}
			{\rm Let $z_1,z\in C$ be feet of $x,y$ on $C$ respectively. Let $\tau:[0,r]\to (M,F)$, $\si:[0,b]\to(M,F)$ be  unit speed minimizing geodesics such that  $\tau(0)=y,\si(0)=x$ and  $\tau(r)=z,\si(b)=z_1$. Here we set $b:=d(x,z_1)=d(x,C)$. Clearly, $\si[0,b]$ intersects $ S^{-1}(C,r)$ at a point, say, $y_1:=\si[0,b]\cap S^{-1}(C,r)$. We then observe that $d(x,y_1)\geq d(x, S^{-1}(C,r))=d(x,y)$ and $d(y_1,z_1)\geq d(y,C)=r$. The triangle inequality then implies $ d(x,C)=d(x,y_1)+d(y_1,z_1)\ge d(x,S^{-1}(C,r))+r= a+r$ and
				$d(x,C)\le d(x,y)+d(y,z)= a+r.$ We therefore have $d(x,y)+d(y,z)=d(x,z)=a+r$ and $d(y_1,z_1)=r.$ We then assert that $x,y$ and $z$ belong to a minimizing geodesic. In fact, let $B(y,\delta(y))$ be a convex ball centered at $y$ and take arbitrary points $x'\in\g(a-\delta(y),a)\cap B(y,\delta(y))$ and $z'\in \tau(0,r)$.	The triangle inequality again implies that $d(x',z')=d(x',y)+d(y,z'),$ and hence the uniqueness of minimizing geodesic joining two points in $B(y,\delta(y))$ implies that $y$ is an interior point of the minimizing geodesic joining $x'$ to $z'$. Therefore $\g(a+r)=\tau(r)=z$. In particular, $z$ and $z_1$ are unique feet of $y$ and $y_1$ of $C$ respectively (see Figure \ 1). }
			\end{proof}
			
			\begin{tikzpicture}[scale=1]
			\draw (-4,0) parabola bend (0,1) (4,0) ;
			\draw (-4.5,2) parabola bend (0,3) (4.5,2) ;
			\draw (0,3) circle (1.5cm);
			\draw[thick][decoration={markings, mark=at position 0.625 with {\arrow{<}}},postaction={decorate}] (0,3) arc (0:60:3cm)node[above] {$x$};
			\draw [thick][decoration={markings, mark=at position 0.825 with {\arrow{>}}},postaction={decorate}](0,3) arc (40:-2:3cm) node[below] {$z=\tau(r)$};
			\draw[dashed][decoration={markings, mark=at position 0.625 with {\arrow{>}}},postaction={decorate}] (-.14,4) arc (50:-11:2.2cm);
			\draw [thick][decoration={markings, mark=at position 0.525 with {\arrow{>}}},postaction={decorate}] (-1.5,5.6) arc (120:215:3.5cm)node[below] {$z_1=\sigma(b)$};
			\node at (-0.2,2.8){$y$};\node at (2,.4){$C$}; \node at (-5.3,2){$S^{-1}(C,r)$}; \node at (-3.4,2.6){$y_1$};\node at (-.4,4){$x'$}; \node at (-.4,5){$\gamma$}; \node at (.9,1.3){$\tau$};\node at (.4,1.8){$z'$};\node at (-3.4,3.3){$\sigma$};\node at (1.2,4.5){$B(y,\delta y)$};\node [below][below=0.4cm] {Figure \ 1};
			\end{tikzpicture}
			
			\medskip
			Let $C\subset(M,F)$ be a closed, connected and locally convex set. We then find an open set $U(C)$ of $M$ and a strong deformation retract $C$ of $U(C)$. In fact, we choose a countable open cover $\{U_i\}_{i=1,2,\dots}$ of $C$ such that for each $i=1,2,\dots$, the closure $\overline{U_i}$ of $U_i$ is compact. Let $\de_i:=\de(\overline{U_i})$ be the convexity radius of $\overline{U_i}$ and set
			\[   U(C):=\bigcup_{i=1}^{\infty}\, B^{-1}(\overline{U_i},\de_i). \]
			If $x\in U(C)$, then we get a number $i$ and a unique foot $f(x)$ of $x$ on $U_i\cap C$. Suppose there is another foot $f'(x)$ of $x$ on $U_i\cap C$. Clearly we have 
			\[  f(x), f'(x)\in B(x,\de_i)\cap C.    \]
			There is a unique minimizing geodesic $T(f(x),f'(x))$ joining $f(x)$ to $f'(x)$ and belonging to $B(x,\de_i)$. The convexity of the distance function from $x$ to $T(f(x),f'(x))$ implies that its minimum is attained at an interior point of $T(f(x),f'(x))$. This is a contradiction to the fact that $T(f(x),f'(x))\subset C$. The above discussion may be summarized as follows:
			\par\medskip
			\begin{Prop}
				{\rm
					A closed locally convex set $C\subset(M,F)$ admits an open set $U(C)$ of  $M$ such that $C$ is a strong deformation retract of $U(C)$. The retraction is given by the foot map $f:U(C)\to C$.
				}
			\end{Prop}
			\begin{Rem} 
				{\rm
					If the local convexity of a closed set $C\subset(M,F)$ is not assumed, then every open set $V\supset C$ may admit a point $x\in V$ with more than one foot on $C$. The backward cut locus $\text{Cut}^{-1}(A)\subset(M,F)$ of a closed set $A\subset(M,F)$ is defined as follows:
					\[
					\text{Cut}^{-1}(A):=\left\{ 
					x\in (M,F) \Bigg\vert \begin{split} 
					& \text{ if} \  f(x) \ \text{ is a foot of} \ x \ \text{on} \ A, \text{then}\ T(x,f(x)) \ \text{is not} \\ & \text{properly  contained in any}  \ T(y,f(y)),\  y\in (M,F) \setminus A, \\ &  \text{where}\  f(y) \ \text{is a foot of} \ y  \ \text{on}\ A. 
					\end{split} 		\right\}
					\]
					
					\noindent This means that $x$ belongs to the backward cut locus to $C$, and hence the foot map $f:V\to C$ may fail to be a retraction map of $C$. Let $\text{Cut}^{-1}(C)\subset(M,F)$ be the backward cut locus to $C$. We then have $d(\text{Cut}^{-1}(C),C)\ge 0$, and equality holds if and only if there is a sequence $\{x_j\}_{j=1,2,\dots}\subset \text{Cut}^{-1}(C)$ such that $\lim_{j\to\infty}\;d(x_j,C)=0$. Let $C\subset(M,F)$ be a compact set whose boundary $\partial C$ is a smooth hypersurface of $(M,F)$. Here we do not assume $C$ to be  locally convex. From the smoothness assumption on $\partial C$, we deduce that $d(\text{Cut}^{-1}(C),C)>0$. Let $U(C)\supset C$ be an open as defined above. Then $C$ has a strong deformation retract and its retraction map is the foot map $f:U(C)\to C$. 
				}
			\end{Rem}
			\par\medskip
			\subsection{The Klingenberg Lemma}\index{Klingenberg Lemma}\index{Lemma!Klingenberg}
			We now summarize important facts on the properties of cut loci on complete Riemannian and Finsler manifolds \cite{IINS}.
			\begin{Prop}\label{Prop:K}
				{\rm
					Let $(M,F)$ be a complete Finsler manifold and $x\in M$ an arbitrary fixed point such that $C(x)\neq\emptyset$. If $y\in C(x)$ and if $\g:[0,d(x,y)]\to(M,F)$ is a minimizing geodesic with $\g(0)=x$, and $\g(d(x,y))=y$, then one of the following holds:
					\begin{enumerate}
						\item $y$ is conjugate to $x$ along $\g$, or 
						\item there exists another minimizing geodesic $\si:[0,d(x,y)]\to(M,F)$ such that $\si(0)=x$ and $\si(d(x,y))=y$.
					\end{enumerate}
				}
			\end{Prop}
			
			\begin{Prop}\label{K1}
				{\rm
					Let $(M,g)$ be a compact Riemannian  manifold and $x\in M$, $y\in C(x)\subset(M,g)$ satisfy  (2) in Proposition \ref{Prop:K}. If
					\begin{equation}\label{eq:cutpoint}
					d(x,y)=i(x)=d(x,C(x)),
					\end{equation}
					then there are exactly two distinct minimizing geodesics $\g, \si:[0,i(x)]\to(M,g)$ such that $\g(0)=\si(0)=x$, $\g(i(x))=\si(i(x))=y$, and $\dot\g(i(x))=-\dot\si(i(x))$.
				}
			\end{Prop}
			\begin{Rem}
				{\rm
					It turns out that Klingenberg's result in  Proposition 4.3 that $\dot\g(i(x))=-\dot\si(i(x))$ does not hold for a Finsler manifold, in general. This phenomenon shows an essential difference between Riemannian and  Finsler geometry.
				}
			\end{Rem}
			\begin{proof}[Proof of Proposition \ref{Prop:K}]
				{\rm
					Suppose that (1) does not hold. Thus $y\in C(x)$ is not conjugate to $x$ along $\g$, and hence there is a small open set $O(y)\subset T_xM$ around $y$ such that $\exp_x|_{O(y)}:O(y)\to(M,F)$ is an embedding. Let $\{\e_j\}_{j=1,2,\dots}$ be a decreasing sequence of positive numbers converging to $0$ and $q_j:=\g(d(x,y)+\e_j)$. For each $j$, let $\tau:[0,a_j]\to(M,F)$ be a minimizing geodesic such that $\tau_j(0)=x, \tau_j(a_j)=q_j$ for $j=1,2,\dots$. Clearly $d(x,q_j)=a_j$ and $\lim_{j\to\infty}a_j=d(x,y)$. Then $\g[0,d(x,y)+\e]$ is not minimizing for all $\e>0$ and hence $\dot\g(0)\neq\dot\tau_j(0)$ for all $j$. Choosing a subsequence, we find a limit of $\{\dot\tau_j(0)\}_j$ and the limit geodesic $\tau:[0,d(x,y)]\to(M,F)$ such that $\tau(0)=x$, $\tau(d(x,y))=\g(d(x,y))$ satisfying $d(x,y)\dot\tau(0)\notin O(y)$.
					This proves that (2) holds.  }	\end{proof}
					Lemma \ref{Lem:BO} below describes certain conditions under which (1) holds. 
					
					The following lemma due to Berger~\cite{B1}  is extended to Finsler manifolds. The proof employs the Riemannian metric in the geodesic polar coordinates as defined in  \S 3.2, and the first variation formula, and is omitted here.
					\begin{Lem}\label{Lem:B1}
						{\rm
							Let $x,y\in (M,F)$. Let $V\subset M$ be an open set around $y$ such that the distance function $d(x,.):V\to \R$ from $x$ attains a  local maximum at $y$. Then there exists for every vector $u\in T_yM,\; u\neq 0$ a minimizing geodesic $\si:[0,d(y,x)]\to(M,F)$ such that
							$\si(0)=x,\,\si(d(y,x))=y$ and 
							\begin{equation}\label{eq:B1}
							g(u,-\dot\si(d(x,y)))\ge 0.
							\end{equation}      
						}
					\end{Lem}
					\par\medskip
					\subsection{The Berger-Omori Lemma}\par
					The following important lemma was first proved  by Berger \index{Berger--Omori Lemma} \index{Lemma!Berger--Omori} for even-dimensional compact Riemannian manifolds of positive sectional curvature with minimal diameter. Then Omori \cite{Omori} proved it for any compact Riemannian manifold with minimal diameter. It is summarized in \cite{Besse}. We prove  the Riemannian version under certain weaker assumptions. The Berger-Omori Lemma can be extended to Finsler manifolds, and has been carried out in \cite{IINS}. \par
					For any vector $u\in TM$  we denote by $\Vert u\Vert$ the Riemannian norm of $u$ .\par
					\begin{Lem}\label{Lem:BO}
						{\rm
							Let $(M,g)$ be a complete Riemannian manifold and $C(x)\neq\emptyset$ for a point $x\in M$. If $y\in C(x)$ satisfies $d(x,y)=i(x)$ and if there exist two distinct minimizing geodesics $\g_0, \g_1:[0,i(x)]\to M$ such that $\g_0(0)=\g_1(0)=x$ and $\g_0(i(x))=\g_1(i(x))=y$ and such that $\dot\g_0(i(x))$ and $\dot\g_1(i(x))$ are linearly independent, then there exists a one-parameter family of minimizing geodesics $\g_t:[0,i(x)]\to M$, $\g_t(0)=x$, $\g_t(i(x))=y$, $0\le t\le 1$ such that
							\begin{equation}\label{eq:luna}
							\dot\g_t(i(x))=\frac{(1-t)\dot\g_0(i(x))+t\dot\g_1(i(x))}{\Vert(1-t)\dot\g_0(i(x))+t\dot\g_1(i(x))\Vert},\quad 0\le t\le 1.
							\end{equation}
							
						}
					\end{Lem}
					
					\begin{proof}
						{\rm
							Let $\rho:=\de(x)/2$ and $\rho':=i(x)-\rho$. We then have
							\[          B(\g_j(\rho'),\rho)\subset B(x, i(x)),\quad j=0,1,      \]
							and also 
							\[  
							\{\bar B(\g_0(\rho'),\rho)\cup\bar B(\g_1(\rho'),\rho)\}\cap\partial B(x,i(x))=\{y\}. 
							\]
							Again, let $\wU_y\subset T_yM$ be the maximal domain in $T_yM$ on which the exponential map at $y$ is an embedding. We already know that $\wC(y)=\partial \wU_y$.  Set
							\[  \mathcal S_j:=(\exp_y|_{U_y})^{-1}(\partial B(\g_j(\rho'),\rho)),\quad j=0,1. \]
							Then $\mathcal S_j$ is a smooth hypersurface in $T_yM$. Clearly, $c_j:=(\exp_y|_{U_y})^{-1}(\g_j[\rho',i(x)])\subset T_yM$ is a straight line segment, and normal to $\mathcal S_j$ at $O\in T_yM$. Using the assumption for $\g_0$ and $\g_1$, we choose a small number $0<r_j<\rho$ and a vector $\xi_j:=-r_j\dot\g_j(i(x))\in T_yM$ such that $B(\xi_j,r_j)\subset (\exp_y|_{U_y})^{-1}(B(\g_j(\rho'),\rho))$, $j=0,1$. In fact, both $\partial B(\xi_j,r_j)$ and $(\exp_y|_{U_y})^{-1}(\partial B(\g_j(\rho'),\rho))$ are smooth hypersurfaces in $T_yM$ and have the same tangent spaces at the origin. We then have
							\[   \exp_y\{\overline{ B}(\xi_0,r_0)\cup\overline{ B}(\xi_1,r_1)\}\cap \partial B(x,i(x))=\{y\}. \]
							
							Let $\xi_\lambda:=(1-\lambda)\xi_0+\lambda\xi_1\in T_yM$ for $\lambda\in[0,1]$. Since the figures are all in a Euclidean space, it is clear that
							\begin{eqnarray}\label{eq:twoballs}
								B(\xi_\lambda,\Vert \xi_\lambda \Vert)\subset B(\xi_0,r_0)\cup B(\xi_1,r_1),\quad \forall\lambda\in[0,1],  \\
								(\exp_y\,|_{U_y})(\partial B(\xi_\lambda,\Vert \xi_\lambda \Vert))\cap\partial B(x,i(x))=\{y\}.
							\end{eqnarray}
							Choosing $h_\lambda>0$ sufficiently small, we deduce from the inclusion relation in $T_yM$ that
							\[ B(\exp_y h_\lambda\xi_\lambda, \Vert h_\lambda\xi_\lambda \Vert)\subset \exp_y(B(\xi_\lambda,\Vert \xi_\lambda \Vert)).    \]  
							Finally, the triangle inequality implies
							\[   d(x,\exp_y h_\lambda\xi_\lambda)+d(\exp_y h_\lambda\xi_\lambda,y)\ge i(x).    \]
							Setting $\si_\lambda:[0,\ell_\lambda]\to(M,g)$ a minimizing geodesic with
							\[   \si_\lambda(0)=x,\quad \si_\lambda(\ell_\lambda)=\exp_y h_\lambda\xi_\lambda,     \]
							we observe from the above triangle inequality that
							\[   d(\si_\lambda(\ell_\lambda),\si_\lambda(i(x))=i(x)-\ell_\lambda\le \Vert h_\lambda\xi_\lambda \Vert.     \]        
							If $\si_\lambda(i(x))\neq y$, then
							\[  \si_\lambda(i(x))\in \bar B(\si_\lambda(\ell_\lambda),\Vert h_\lambda\xi_\lambda \Vert).    \] 
							On the other hand, $\si_\lambda(i(x))\in\partial B(x,i(x))$ means that the point $\si_\lambda(i(x))$ stays outside $\exp_y(B(\xi_0,\Vert\xi_0\Vert)\cup B(\xi_1,\Vert\xi_1\Vert))\setminus\{y\}$, which is a contradiction. Thus we have $\si_\lambda(i(x))=y$ for all $\lambda\in[0,1]$.     
						}
					\end{proof}
					\begin{Rem}
						{\rm
							To discuss the Berger-Omori Lemma in the extent of Finsler manifolds, we need to consider $T_yM$ as a normed space and introduce a new idea for the proof of the Finsler version; see Theorem 4.2 in \cite{IINS} for details.
						}
					\end{Rem}
					\par\medskip
					\subsection{The Rauch conjecture}
					The classical Rauch conjecture~\cite{Rauch} predicts that the cut locus $\wC(x)\subset T_xM$ and the conjugate locus $\wJ(x)\subset T_xM$ to a point $x$ on a compact and simply connected Riemannian $n$-manifold 
					$(M,g)$ have a point in common in its tangent space $T_xM$:
					\begin{equation}\label{eq:Rauch}
					\wC(x)\cap\wJ(x)\neq\emptyset,\quad\text{for all $x\in M$}.            
					\end{equation}                
					The conjecture is true for all metrics on $S^2$ and for all compact rank one symmetric spaces.  Also it is true for a complete noncompact Riemannian $2$-manifolds homeomorphic to $\R^2$. A.Weinstein~\cite{Weinstein} settled this conjecture negatively in general, by proving that if $M$ is not homeomorphic to $S^2$, then there exists a metric $g$ and a point $x\in M$ such that 
					\[          \wC(x)\cap\wJ(x)=\emptyset.      \]  
					
					We next discuss the results obtained in \cite{ISS} on complete Finsler $n$-manifolds. Let $(M,F)$ be a connected and geodesically complete Finsler $n$-manifold, $n\ge 2$.  It is elementary to see that $i_x(u)\le c_x(u)$, in general, and $i_x(u)=c_x(u)<\infty$ holds for some $u\in\Si_x$ if and only if the Rauch conjecture is valid at $x$. 
					If $\wC_x=\emptyset$ (or equivalently, $C(x)=\exp_x\,\wC_x=\emptyset$), then $\g_u:[0,\infty)\to (M,F)$ is a ray for all $u\in\Si_x$. Such a point is called a {\it forward pole} of $(M,F)$.  An ellipsoid with foci at $x,y\in (M,F)$ and radius $r>d(x,y)$ is denoted by $E_{xy}(r)\subset M$ and
					\[   E_{xy}(r):=\{ z\in M | d(x,z)+d(z,y)=r \},        \] 
					\[    B_{xy}(r):=\{ z\in M | d(x,z)+d(z,y)<r \}.     \] 
					Notice that $B_{xx}(2r)\neq B(x,r)$ follows if $d$ is not symmetric.  We further define the function $F_{xy}:(M,F)\to\R$ by
					\[        F_{xy}(z):=d(x,z)+d(z,y),\quad z\in M.       \]   
					Notice also that $F_{xy}\neq F_{yx}$ and $E_{xy}(r)\neq E_{yx}(r)$.    
					\par\medskip
					\begin{Th}[compare Theorem 1 ; \cite{ISS}]\label{Th:GAFA}
						{\rm
							Let $(M,F)$ be an $n$-dimensional compact Finsler manifold. Assume that there is a point $x\in M$ satisfying 
							\begin{equation}\label{eq:Weinstein}
							\wC(x)\cap\wJ(x)=\emptyset.
							\end{equation}            
							Then for any point $y \in M\setminus \{x\}$ there exist at least two distinct geodesics emanating from $x$ and ending at $y$.  
						}
					\end{Th}

					\begin{proof}
						{\rm
							If there exist two distinct minimizing geodesics emanating from $x$ and ending at $y$, nothing is left to prove. \par
							If $y\in C(x)$ and if there exists a unique minimizing geodesic $\tau:[0,d(x,y)]\to (M,F)$ with $\tau(0)=x$ and $\tau(d(x,y))=y$, then every extension $\tau|_{[0, d(x,y)+\e]}$ of $\tau$ beyond $y$ is not minimizing for all $\e>0$. The uniqueness of $\tau$ implies that if $\tau_\e:[0,d(x,\tau(d(x,y)+\e)]\to(M,F)$ is a minimizing geodesic with $\tau_\e(0)=x$ and $\tau_\e(d(x,\tau(d(x,y)+\e)))=\tau(d(x,y)+\e)$, then $\tau_\e\neq\tau|_{[0,d(x,y)+\e]}$. The uniqueness of a minimizing geodesic from $x$ to $y$  then implies that $\lim_{\e\to 0}\tau_\e=\tau$. Therefore $\exp_x:T_xM\to(M,F)$ is not bijective in any open set of $d(x,y)\cdot\dot\tau(0)$, and hence $y$ is conjugate to $x$ along $\tau$. Thus we have $d(x,y)\cdot\dot\tau(0)\in TC(x)\cap TJ(x)$, a contradiction to (\ref{eq:Weinstein}). Thus we observe that there are at least two distinct minimizing geodesics from $x$ to $y$ for all $y\in C(x)$.  We may therefore suppose that there exists a unique minimizing geodesic from $x$ to $y$ and that $y\notin C(x)$.
							\par\medskip
							The construction of a non-minimizing geodesic joining $x$ to $y$ is achieved by using the technique developed in~\cite{ISS}. There exists a cut point $x_0\in C(x)$ such that, 
							\begin{equation}\label{eq:ellipsoid}  
							B_{xy}(\ell)\cap C(x)=\emptyset,\; x_0\in E_{xy}(\ell)\cap C(x),\; \ell:=F_{xy}(x_0).          
							\end{equation}
							Let $\si:[0,d(x_0,y)]\to(M,F)$ be a minimizing geodesic with 
							$\si(0)=x_0$, $\si(d(x_0,y))=y$. We assert that {\it there exist exactly two distinct minimizing geodesics $\g_j:[0,d(x,x_0)]\to(M,F)$, $j=1,2$, with $\g_j(0)=x$, $\g_j(d(x,x_0))=x_0$ such that one of the composed geodesics $\g_1\ast\si$ or $\g_2\ast\si$ forms a geodesic joining $x$ to $y$}. Here, $\g_j\ast\si:[0,\ell]\to(M,F)$ is a broken geodesic given by
							\begin{equation*}
								\g_j\ast\si(t)=\begin{cases}
									\g_j(t),\quad t\in[0,d(x,x_0)],  \\
									\si(t-d(x,x_0)),\quad t\in[d(x,x_0),\ell].
								\end{cases}
							\end{equation*}                    
							
							For the proof of the above assertion, we argue  by deriving a contradiction, assuming the contrary.  Suppose that there are more than two distinct minimizing geodesics from $x$ to $x_0$. If $\si$ is a fixed minimizing geodesic from $x_0$ to $y$, we  choose two minimizing geodesics $\g_1$ and $\g_2$ such that both $\g_1\ast\si$ and $\g_2\ast\si$ are broken geodesics with a corner at $x_0$. We then find points $z\in\si(0,d(x_0,y)]$ and $y_j\in\g_j([0,d(x,x_0))$ lying in a convex ball around $x_0$. The short cut principle then implies that
							\[    d(y_j,z)<d(y_j,x_0)+d(x_0,z),\quad j=1,2,     \]
							and hence 
							\begin{align*}
								F_{xy}(z) &=d(x,z)+d(z,y)<((d(x,y_j)+d(y_j,z))+d(z,y)  \\
								&<  d(x,y_j)+(d(y_j,x_0)+d(x_0,z))+d(z,y)  \\
								&= d(x,x_0)+d(x_0,y)=F_{xy}(x_0)=\ell.
							\end{align*}                  
							We therefore have $z\in B_{xy}(\ell)$, and in particular, $z\notin C(x)$. This implies that $\si(0,d(x_0,y))\subset\exp_x(U_x)$. From (\ref{eq:Weinstein}), there exists an open set $\Om_i\in T_xM$ of $d(x,x_0)\cdot\dot\g_j(0)$ for $j=1,2$ such that $\exp_x|_{\Om_j}:\Om_j\to(M,F)$ is an embedding. The lifting of $\si((0,d(x_0,y)])$ along the diffeomorphism $\exp_x|_{U_x}:U_x\to (M,F)\setminus C(x)$ forms distinct curves joining $d(x,x_0)\cdot\dot\g_j(0)$ to $d(x,y)\cdot\dot\si(0)$, $j=1,2$. However this is impossible by the uniqueness of $d(x,y)\cdot\dot\si(0)$.\par\smallskip
							As shown in the proof of the above assertion, one of the $\g_1\ast\si$ and $\g_2\ast\si$ forms a geodesic emanating from $x$ and ending at $y$; it may be noted that  since this geodesic passes through a forward cut point $x_0$ to $x$, it is not minimizing.
						}
					\end{proof}
					As an application of Theorem \ref{Th:GAFA}, we deduce the following:
					
					\begin{Th}\label{Th:injectivity}
						{\rm
							Let $(M,F)$ be a compact Finsler $n$-manifold, $n\ge 2$	{\it and $\lambda=\lambda(M)$ be the reversibility constant of $M$, as defined in (2.1).} If $x\in M$ is such that the Rauch conjecture does not hold, then there exists a geodesic loop $\g_x : [0,2\ell_x]\to (M,F)$ with $\g_x(0)=\g_x(2\ell_x)=x$ such that 
							\[       (1+\lambda^{-1})i(x)\le L(\g_x).    \]
						}
					\end{Th}
					\begin{proof}
						The same technique as involved in the proofs of Theorem \ref{Th:GAFA} is employed here. For an arbitrary fixed point $y\notin C(x)$ with $y\neq x$, we choose a point $x_0(y)\in C(x)$ such that, setting $\ell(y):=d(x,x_0(y))+d(x_0(y),y)$, 
						\[    B_{xy}(\ell(y))\cap C(x)=\emptyset,\quad x_0(y)\in E_{xy}(\ell(y))\cap C(x). \]
						By Theorem \ref{Th:GAFA}, there exists a geodesic 
						$\g_{xy}:[0,\ell(y)]\to M$ such that $\g_{xy}(0)=x$, $\g_{xy}(d(x,x_0(y))=x_0(y)$ and $\g_{xy}(\ell(y))=y$. From the construction we see that $L(\g_{xy}[0,d(x,x_0(y))])=d(x,x_0(y))$ and $L(\g_{xy}[d(x,x_0(y)),\ell(y)])=d(x_0(y),y)$. Taking a sequence $\{y\}\subset M\setminus C(x)$ converging to $x$, we find a geodesic loop $\g_x:[0,\ell_x]\to M$ as the limit of $\{\g_{xy}\}$, where $\g_x=\lim_{y\to x}\g_{xy}$ and 
						$\ell_x=\lim_{y\to x}\ell(y)$. We then observe that $\g_x(d(x,x_0(x)))\in C(x)$ and that 
						\[     L(\g_{xy}[0,d(x,x_0(y))])\ge i(x).    \]
						From the triangle inequality we have,  
						\begin{align*}
							L(\g_{xy}[d(x,x_0(y)),\ell(y)]) &=d(x_0(y),y)\ge \lambda^{-1}d(y,x_0(y)) \\
							&\ge \lambda^{-1}[d(x,x_0(y))-d(x,y)]\ge\lambda^{-1}(i(x)-d(x,y)).    
						\end{align*}      
						Therefore, by letting $y\to x$ we get, $\ell_x\ge(1+\lambda^{-1})i(x)$.  This proves Theorem~\ref{Th:injectivity}.
					\end{proof}
					\par\medskip
					\subsection{Poles}
					The original Rauch conjecture\index{Rauch conjecture}\index{conjecture!Rauch} was considered on compact and simply connected Riemannian manifolds. We discuss it on complete non-compact Riemannian and Finsler manifolds  admitting poles. For the discussion of the Rauch conjecture on complete non-compact manifolds, we need the notion of poles. The Rauch conjecture is valid for a point $x\in (\mathbb R^2,g)$ of a complete noncompact Riemannian $2$-manifold homeomorphic to a plane, if $C(x)\neq\emptyset$. In fact, $C(x)$ for every point $x\in(\mathbb R^2,g)$ carries the structure of a tree. A cut point $y\in C(x)$, $x\in(\mathbb R^2,g)$ is called an {\it endpoint} of $C(x)$ if $C(x)\setminus \{x\}$ is connected. The cut loci of Riemannian $2$-manifolds have been discussed by many authors; for instance, see \cite{Myers1}, \cite{Myers2}, \cite{Whitehead2} and \cite{ST1}.  Let $x\in(\mathbb R^2,g)$ be a point such that $C(x)\neq\emptyset$. Then there is a point $y\in C(x)$ which is an endpoint of $C(x)$. If $\g:[0,d(x,y)]\to(\mathbb R^2,g)$ is a minimizing geodesic with $\g(0)=x$, $\g(d(x,y))=y$, then $y$ is conjugate to $x$ along $\g$, and hence $\wC(x)\cap\wJ(x)$ contains $d(x,y)\dot\g(0)$, if $C(x)\neq\emptyset$. \par
					Let $(M,F)$ be a geodesically complete Finsler $n$-manifold. A unit speed forward geodesic $\g:[0,\infty)\to(M,F)$ is by definition a {\it forward ray} if every subarc $\g|_{[a,b]}$, $0\le a<b<\infty$ of $\g$ is minimizing. A point $x\in(M,F)$ is called a {\it forward pole} if $C(x)=\emptyset$. Clearly, $\exp_x:T_xM\to(M,F)$ is a diffeomorphism if and only if $x$ is a forward pole. A backward geodesic $\g^{-1}(-\infty,0]\to(M,F)$ is called a {\it backward ray} if
					\[   d(\g^{-1}(s),\g^{-1}(t))=t-s,\quad\text{for all $0>t>s>-\infty$}.     \] 
					A point $y\in(M,F)$ is called a {\it backward pole} if every backward geodesic $\si^{-1}:(-\infty,0]\to(M,F)$ with $\si^{-1}(0)=y$ is a backward ray.\par\medskip
					The relation between the Rauch conjecture and poles on complete noncompact Riemannian $n$-manifolds has been discussed in \cite{ISS}. We have the following relation between the Rauch conjecture and poles on complete non-compact Finsler manifolds. The proof is essentially contained in \cite{ISS} and is omitted here.\par\medskip
					\begin{Prop}\label{Prop:pole}
						{\rm
							Let $(M,g)$ and $(M,F)$ be complete Riemannian and Finsler manifolds respectively. 
							
							\begin{enumerate}
								\item If $(M,g)$ admits a pole at $y\in M$ and if $x\in M$ is not a pole, then the Rauch conjecture is valid at $x$.
								\item  If $(M,F)$ admits a backward pole, then either the Rauch conjecture holds at a point $y\in M$ or otherwise, $y$ is a forward pole.
							\end{enumerate}
						}
					\end{Prop}
					
					\par\medskip
					\subsection{The continuity of the injectivity radius function}
					We discuss the continuity of injectivity radius functions on  complete Finsler manifolds. The compactification $[0,\infty]:=[0,\infty)\cup\{\infty\}$ of the half line is employed here. \par
					\begin{Lem}\label{Lem:injectivityradius}
						{\rm
							Let $(M,F)$ be a complete Finsler manifold. The injectivity radius function $i:M\to[0,\infty]$ is continuous at every point $x \in M$ where $i(x)<\infty$.
						}
					\end{Lem}
					\begin{proof}
						{\rm
							Let $x\in M$ and $\{x_j\}_{j=1,2,\dots}\subset M$ be such that $\lim_{j\to\infty}x_j=x$. Let $\{y_j\}_{j=1,2,\dots}\subset M$ be chosen such that $d(x_j,y_j)=i(x_j)=:\ell_j$ for all $j=1,2,\dots$. Let $\g_j:[0,\ell_j]\to(M,F)$ be a minimizing geodesic with $\g_j(0)=x_j$ and $\g_j(\ell_j)=y_j$ for all $j=1,2,\dots$. In view of Proposition \ref{Prop:K}, by choosing subsequence if necessary, we have two cases:\par
							{\bf Case 1}. Assume that $y_j$ is conjugate to $x_j$ along $\g_j$. Setting $v_j:=\ell_j\dot\g_j(0)$, we have
							\begin{equation}\label{eq:*}
							{\rm det}(d(\exp_{x_j})_{v_j})=0,\quad {\rm for\ all \ } j=1,2,\dots     
							\end{equation}  
							Thus we observe that $\lim_{j\to\infty}y_j=y$ is a conjugate point to $x$ along $\g$, where $\g$ is defined by the limit: $v=\dot\g(0):=\lim_{j\to\infty}\dot\g_j(0)$, and hence $i(x)\le F(x,v)=\lim_{j\to\infty}i(x_j)$.\par
							
							{\bf Case 2}. We now assume that there exist minimizing geodesics $\g_j:[0,\ell_j]\to(M,F)$ emanating from $x_j$ and ending at $y_j$ such that $i(x_j)=d(x_j,y_j)$ and $y_j\in C(x_j)$ is not conjugate to $x_j$ along $\g_j$ for all $j=1,2,\dots$. From Proposition \ref{Prop:K}, we get that there are exactly two minimizing geodesics $\g_j, \si_j:[0,\ell_j]\to(M,F)$ such that $\g_j(0)=\si_j(0)=x_j$ and $\g_j(\ell_j)=\si_j(\ell_j)=y_j$. Choosing a subsequence, if necessary, we get limit geodesics $\g:=\lim_{j\to\infty} \g_j$ and $\si:=\lim_{j\to\infty}\si_j$ together with $\ell:=\lim_{j\to\infty}\ell_j$. Clearly $\g$ and $\si$ are distinct minimizing geodesics emanating from $x$ and ending at $y$, and hence $y\in C(x)$.  Therefore
							\[    d(x,y)=\lim_{j\to\infty}\ell_j\ge i(x).     \]
							This proves the lower semi-continuity of the injectivity radius function $i$. \par
							We finally conclude the proof in this case by showing that
							\[        \lim_{j\to\infty}\ell_j\le\ell.      \]
							Suppose to the contrary that there exists a point $x$ and a sequence $\{x_j\}$ converging to $x$ such that 
							\begin{equation}\label{eq:contradiction}
							\lim_{j\to\infty}\,\ell_j>\ell.
							\end{equation}
							Using (\ref{eq:contradiction}) we choose a sufficiently small positive number 
							\[      \e:=\lim_{j\to\infty}(\ell_j-\ell)/2.         \]             
							Let $U_j:=U_{x_j}\subset T_{x_j}M$ be the domain such that $\partial U_j=\wC(x_j)$. There is a large number $j_0$ such that, 
							\[      \bar B_j(O,\ell+\e)\subset U_j,\quad{\rm for\ all}  \ j>j_0.       \]    
							Here we set $\bar B_j(O,r)\subset T_{x_j}M$ an $r$-ball centered at the origin $O\in T_{x_j}M$. Then $\exp_{x_j}|_{\bar B_j(O,\ell+\e)}:\bar B_j(O,\ell+\e)\to B(p_j,\ell+\e)$ is a smooth embedding. From the continuity of $\Pi:TM\to M\times M$, it follows that
							\[  \exp_x|_{U_x}:\bar B(O,\ell+\e)\to B(p,\ell+\e)     \]
							is an embedding, and hence $\ell:=i(x)\ge\ell+\e$, a contradiction.        
						}
					\end{proof}
					
					For the proof of the continuity of the injectivity radius function, where $M$ is non-compact, we now only need to prove that it is continuous at any point where $i(x)=\infty$. This is achieved by the following:
					\begin{Lem}\label{Lem:infinity}
						{\rm
							Let $(M,F)$ be a complete non-compact Finsler manifold. Then the injectivity radius function $i:(M,F)\to[0,\infty]$ is continuous at any point $x\in M$ where $i(x)=\infty$.
						}
					\end{Lem}
					\begin{proof}
						{\rm
							Let $\{x_j\}_{j=1,2,\cdots}$ be a sequence of points converging to $x$. We then prove that 
							\[       \lim_{j\to\infty}\,i(x_j)=\infty.           \]
							Suppose contrary that there exists a sequence of points $\{x_j\}$ converging to $x$ such that $\lim_{j\to\infty}\,i(x_j)<\infty$.\par
							The same notations as in the previous Lemma \ref{Lem:injectivityradius} will be used. Let $y:=\lim_{j\to\infty}y_j$, where $y_j$, for every $j$, is conjugate to $x_j$ along $\g_j$. We observe that $y$ is conjugate to $x$ along $\g:=\lim_{j\to\infty}\g_j$. However this is a contradiction to $i(x)=\infty$.\par
							Now suppose  that $y_j$ for each $j$ is not conjugate to $x_j$ along $\g_j$. We then have two minimizing geodesics $\g_j, \si_j:[0,\lim_{j\to\infty}\ell_j]\to(M,F)$ joining $x_j$ to $y_j$. If $\g:=\lim_{j\to\infty}\g_j$ and $\si:=\lim_{j\to\infty}\si_j$, then $\g$ and $\si$ are distinct minimizing geodesics from $x$ to $y:=\lim_{j\to\infty} y_j$. Therefore $y\in C(x)$, contradicting to $C(x)=\emptyset$.
						}
					\end{proof}
					\subsection{Pointed Blaschke manifolds}\par
					The Riemannian Blaschke manifolds\index{Riemannian Blaschke manifold} have been fully investigated by Berger and his colleagues and the findings are summarized in \cite{Besse}.
					Instead of setting down the curvature assumption, a certain restriction on the diameter and injectivity radius of a compact Finsler manifold
					is proposed in this respect. Let $(M,F)$ be a compact Finsler $n$-manifold. We have discussed the Finsler version of the fundamental properties of cut locus and conjugate locus. The diameter $d(M)$ of $(M,F)$ is defined by
					\[      d(M):=\max\,\{d(x,y)\,|\, x, y\in M\}.           \]
					The injectivity radius $i(M)$ of $(M,F)$ is defined by
					\[      i(M):=\min\,\{i(x)\,|\, x\in M\}.       \]
					
					\begin{Def}
						{\rm 
							A Finsler manifold $(M,F)$ is called  a {\it Blaschke Finsler manifold} if 
							\begin{equation}\label{eq:BFM}
							d(M)=i(M),
							\end{equation}              
							and $(M,F)$ is called a {\it pointed Blaschke manifold with a base point at $x\in M$} if  
							\begin{equation}\label{eq:pBFM}
							i(x)=\max\,\{d(x,y)\,|\, y\in M\}.     
							\end{equation}
							Such a pointed Blaschke manifold with a base point $x$ is denoted by $(M,F:x)$.     
						}      
					\end{Def}
					We refer to \cite{Besse} for a discussion on Riemannian Blaschke manifolds. Clearly (\ref{eq:pBFM}) holds at each point of $M$ if (\ref{eq:BFM}) is satisfied. A classical result by Berger and Klingenberg states that if $(M,g)$ is a compact simply connected Riemannian manifold whose sectional curvature ranges over $[1/4,1]$,
					then $i((M,g))\ge\pi$; see \cite{B1}, \cite{B2}, and \cite{Klingenberg1}. Moreover $M$ is homeomorphic to $\bS^n$ if $d((M,g))>\pi$, and isometric to one of the symmetric spaces of  compact type if $d((M,g))=i(M,g))=\pi$.
					\par\medskip                       
					We set, for simplicity, $\ell:=i(x)$ for a pointed Blaschke Finsler manifold $(M,F:x)$. Then every cut point $y\in C(x)$ has the property that 
					$d(x,y)=\ell$ and that $y$ is the farthest point from $x$. Therefore we have the assumptions in Proposition \ref{Prop:K} and Lemma \ref{Lem:BO}.  Let 
					$y\in C(x)$ and set
					\[    \G_{xy}:=\{ \g:[0,\ell]\to(M,F:x)\,|\, \g(0)=x,\;\g(\ell)=y\},     \]
					and further set
					\[        A_{xy}:=\{ \dot\g(\ell)\,|\, \g\in\G_{xy}\}.       \]
					We then observe from Lemma \ref{Lem:foot} that 
					\begin{equation}\label{eq:fibering}
					\partial B(x,r)=\partial B^{-1}(C(x),\ell-r), \quad\text{for all}\quad  r\in(0,\ell).
					\end{equation}   
					
					The discussion on pointed Blaschke Finsler manifolds is divided into two
					cases, according to whether the manifold is simply connected or non-simply connected. We first discuss a simpler case, roughly speaking, where $(M,F:x)$ does not satisfy the Rauch conjecture at $x$. 
					\begin{Lem}\label{Lem:loop}
						{\rm
							Assume that the Rauch conjecture is not valid at the base point $x\in(M,F:x)$, i.e., $\wC(x)\cap\wJ(x)=\emptyset$. Then every point $y\in C(x)$ is joined to $x$ by exactly two distinct minimizing geodesics
							\[   \g, \si:[0,\ell]\to(M,F:x),\; \g(0)=\si(0)=x,\; \g(\ell)=\si(\ell)=y,   \]
							such that $y$ is not conjugate to $x$ along them.  
						}
					\end{Lem}
					Lemma \ref{Lem:loop} is a direct consequence of Lemma \ref{Lem:BO} of Berger-Omori. Its proof is omitted. We observe from Lemma~\ref{Lem:loop} that there exists a fixed-point free involution $\psi$ on $\Si_x$ such that $\psi(\dot\g(0))=\dot\si(0)$. Clearly we have 
					\[ \exp_x\,\ell u=\exp_x\,\ell\psi(u)\in C(x), \quad\text{for all}\quad  u\in\Si_x. \]   
					\par\medskip
					Summing up the above discussion we have the following topological conclusion. \par            
					\begin{Th}\label{Th:simplestcase}
						{\rm
							Let $(M,F:x)$ is a pointed Blaschke-Finsler manifold with base point $x$. If the Rauch conjecture is not valid at the base point $x\in(M,F:x)$, then we have
							\begin{enumerate}
								\item the cut locus to $x$ is a smooth hypersurface and diffeomorphic to the quotient space $\Si_x/\{\psi : \psi^2={\rm Id.}\}$, i.e., the cut locus is homeomorphic to a real projective space;
								\item the universal cover $\widetilde M$ of $M$ is homeomorphic to $\bS^n$ and $M$ is homeomorphic to the real projective space;
								\item the fundamental group of $M$  is cyclic of order two.
							\end{enumerate}
						}
					\end{Th}
					\begin{Rem}
						{\rm
							The assumption in Theorem \ref{Th:simplestcase} is too strong. In fact we prove $\wC(x)\cap\wJ(x)=\emptyset$ if $(M,F:x)$ satisfies
							\[   \wC(x)\setminus\wJ(x)\neq\emptyset.       \]
						}
					\end{Rem}
					\noindent In the Riemannian case, $\g$ and $\si$ together form a simple closed geodesic loop at $x$, and $\wJ(x)$ is a $2\ell$-sphere and $J(x)=\{x\}$.
					\medskip
					In view of Theorem \ref{Th:simplestcase}, we observe that if $(M,F:x)$ is a simply connected pointed Blaschke Finsler manifold, then the Rauch conjecture is valid at $x$. Moreover, if $(M,F:x)$ is simply connected, we have
					\[    \wC(x)=\wJ(x).     \]
					The Berger-Omori Lemma \ref{Lem:BO} implies that $A_{xy}$ is a convex set. Moreover the multiplicity of the conjugate point $y$ to $x$ along a minimizing geodesic $\g:[0,\ell]\to M$ with $\g(0)=x$, $\g(\ell)=y$ is independent of the choice of $\g\in\G_{xy}$. Since the dimension of the convex sets $A_{xy}$ is lower semi-continuous on $C(x)$, it is constant on $C(x)$. Therefore the rank of the exponential map $\exp_x$ at each point of $C(x)$ is constant. Hence
					the implicit function theorem implies that $\exp_x\,\wC(x)=C(x)$ is a compact smooth submanifold of $M$. Thus the set of all points on minimizing geodesics belonging to $\G_{xy}$ forms a $k$-dimensional submanifold homeomorphic to $\bS^k$, where $k$ is the dimension of $A_{xy}$, $y\in C(x)$. It follows from the relation (\ref{eq:fibering}) that
					$B(x,r)$ for $r\in(0,\ell)$ is simply covered by $(k-1)$-dimensional spheres and its quotient space is nothing but $C(x)$. We still have much more discussion to complete this case.\par\medskip

					\section{Busemann functions and convex functions}\par\medskip
					
					\subsection{Busemann functions}\index{Busemann function}
					We discuss forward rays and forward Busemann functions on complete non-compact Finsler manifolds. \par
					The  definition of a Busemann function is found in \S 22 of \cite{Busemann}.  A forward Busemann function $F_\g:(M,F)\to\R$ for a ray $\g:[0,\infty)\to(M,F)$ is defined as follows:
					\[     F_\g(x):=\lim_{t\to\infty}(t-d(x,\g(t))), \quad x\in M.     \]
					A backward ray $\g:(-\infty,0]\to(M,F)$ and a backward Busemann function for the backward ray $\g$ are similarly defined. A {\it super Busemann function\index{Super Busemann function}\index{Busemann function!super} $F_x:(M,F)\to\R$ at $x$} is defined by
					\[   F_x(y):=\sup\,\{\,F_\g(y)\,|\,\text{$\g$ is a forward ray with 
						$\g(0)=x$}\},\quad y\in M.   \] \par 
						Clearly, the function $t\to (t-d(x,\g(t)))$ is monotone increasing in $t$ and bounded above by $d(x,\g(0))$. Thus $F_\g$ is well defined, for $t-d(x,\g(t))$ converges uniformly on a compact set and $F_\g$ is locally Lipschitz continuous. A unit speed forward ray $\si:[0,\infty)\to(M,F)$ is by definition {\it asymptotic to $\g$} if there exists a sequence of unit speed minimizing geodesics $\{\si_j:[0,\ell_j]\to(M,F)\}_{j=1,2,_{\cdots}}$ such that $\lim_{j\to\infty}\dot\si_j(0)=\dot\si(0)$, $\si_j(\ell_j)=\g(t_j)$ for a monotone divergent sequence $\{t_j\}$. 
						The asymptotic relation is in general neither symmetric nor transitive. If $(M,g)$ is a complete and simply connected Riemannian manifold of non-positive sectional curvature, then the asymptotic relation between two rays $\a, \beta:[0,\infty)\to(M,g)$ satisfies the following inequality:
						\[     d(\a(t),\beta(t))\le d(\a(0),\beta(0)),\quad\text{for all $t\ge 0$}.   \]
						Therefore only in this case the asymptotic relation is an equivalence relation.\par\medskip   
						The sequence of points $\{\si_j(0)\}_{j=1,2,_{\cdots}}$ cannot be chosen to be a point $\si(0)$, as is seen in the following example.
						\begin{Ex}
							{\rm
								Let $\mathcal F\subset\mathbb R^3$ be a rotation surface of parabola in a Euclidean $3$-space. Let $\{(r,\theta)\,|\, r>0, \theta\in[0,2\pi)\}$ be the geodesic polar coordinate system around the pole $(0,0)$, and $\g_\theta:[0,\infty)\to\mathcal F$ for $\theta\in[0,2\pi]$ be the meridian $\g_\theta(r):=(r,\theta)$, $r\ge 0$. We observe that all the meridians are asymptotic to each other. In fact, let $\theta_0\in[0,2\pi)$ be an arbitrary fixed number and $\{\theta_j\}_{j=1,2,_{\cdots}}\subset[0,2\pi)$ be a monotone sequence with $\lim_{j\to\infty}\theta_j=\theta_0$. Let $\{r_j\}_{j=1,2,_{\cdots}}$ be a monotone decreasing sequence of positive numbers with $\lim_{j\to\infty}\,r_j=0$. If we set $\g_j(t):=\g_{\theta_j}(t+r_j),\, t>0$, then $\g_j$ for each $j=1,2,_{\cdots}$ is asymptotic to $\g_0$, and hence so is $\g_{\theta_0}=\lim_{j\to\infty}\g_j$.   
							}
						\end{Ex}
						
						Assuming that a ray $\si:[0,\infty)\to(M,F)$, $y:=\si(0)$ is asymptotic to another ray $\g:[0,\infty)\to(M,F)$, $x:=\g(0)$, we say that $\si$ is a {\it maximal asymptotic ray to $\g$ if $\si$ is not properly contained in any ray which is asymptotic to $\g$}. We also say that a ray is {\it maximal} if and only if it is not properly contained in any ray. A long-standing open problem proposed by Busemann in \cite{BI} is stated as follows:\par
						\begin{Pb}
							{\rm
								Is a maximal asymptotic ray a maximal ray?
							}
						\end{Pb}   
						
						This problem was solved in the negative by Innami in \cite{Innami3} by exhibiting an example of a surface in $\mathbb R^3$ on which there is a maximal asymptotic ray which is not a maximal ray.\par 
						The local Lipschitz property (1) in Proposition \ref{Prop:Busemann} of $F_\g$ implies that it is differentiable almost everywhere. Then (6) in Proposition \ref{Prop:Busemann} shows that $F_\g$ is differentiable at an interior point of some asymptotic ray to $\g$. Let $\si(0)$ be the initial point of a maximal asymptotic ray to $\g$. If there exists a unique asymptotic ray to $\g$ passing through $\si(0)$, we may view $\g(0)$ and $\g(\infty)$ as being conjugate pair along $\g$, (and this corresponds to (1) in Proposition \ref{Prop:K}). Otherwise, there exists another ray $\si_1:[0,\infty)\to(M,F)$ which is asymptotic to $\g$. Therefore we may view the set of all the initial points of rays asymptotic to $\g$ as the cut locus to a point at infinity obtained by $\g(\infty)$, (and this corresponds to (2) in Proposition \ref{Prop:K}). If $F_\g$ attains its minimum at a point $x\in(M,F)$, then there exists for every unit vector $u\in\Si_x$ a ray $\si:[0,\infty)\to(M,F)$ asymptotic to $\g$ such that $g_u(u,\dot\si(0))\ge 0$. This corresponds to Lemma 4.2. \par\medskip
						\subsection{Properties of Busemann functions on $(M,F)$}
						We denote by $F^{-1}_\g(\{a\})$ and $F^{-1}_\g(-\infty,a]$ the $a$-level set and the $a$-sublevel set of $F_\g$ respectively. The basic properties of Busemann functions\index{Busemann function} are stated in \S\S 22 and  23 of \cite{Busemann} and those on  complete Finsler and Riemannian manifolds $(M,F)$ are summarized in \cite{Shiohama0} and \cite{Shiohama} as follows: \par
						\begin{Prop}[Properties of Busemann functions]\label{Prop:Busemann}\par
							{\rm
								Let $\g:[0,\infty)\to(M,F)$ be a forward ray and $F_\g:(M,F)\to\mathbb{R}$ a Busemann function for $\g$. We then have:
								\begin{enumerate}
									\item $F_\g$ is locally Lipschitz.
									\item A level set $F^{-1}_\g(\{a\})$ for $a\in F_\g(M)$ is obtained by     
									\[      F^{-1}_\g(\{a\})=\lim_{t\to\infty}S^{-1}(\g(t),t-a).         \]
									\item If $a,b\in F_\g(M)$ satisfies $a\le b$, then
									\[    F^{-1}_\g(-\infty,a]=\{y\in F_\g^{-1}(-\infty,b]\,|\, d(y,F^{-1}_\g(\{b\})\ge b-a\},  \]
									and 
									\begin{gather*}
										F^{-1}_\g(\{a\})=\{y\in F^{-1}_\g(-\infty,b])\,|\, d(y,F^{-1}_\g(\{b\}))=b-a\} \\
										=S^{-1}(F^{-1}_\g(\{b\}),b-a)\cap F^{-1}_\g(-\infty,b].
									\end{gather*}
									\item A unit speed geodesic $\si:[0,\infty)\to(M,F)$ is a forward ray asymptotic to $\g$ if and only if 
									\[   F_\g\circ\si(t)=t+F_\g\circ\si(0),\quad {\rm for\ all}\ t\ge 0.   \]
									\item If $x\in M$ and $a\in F_\g(M)$ satisfy $a>F_\g(x)$, and if $\si:[0,\ell]\to(M,F)$ is a unit speed minimizing geodesic with $\si(0)=x$ such that $\si(\ell)$ is a foot of $x$ on $F^{-1}_\g(\{a\})$, then the 
									extension $\si:[0,\infty)\to(M,F)$ of $\si$ is a forward ray asymptotic to $\g$.
									\item $F_\g$ is differentiable at a point $x\in M$ if $x$ is an interior of some ray asymptotic to $\g$. 
								\end{enumerate}
							}
						\end{Prop}
						The detailed proof of Proposition \ref{Prop:Busemann} on Riemannian manifolds was given in \cite{Shiohama} and the same proof technique for the Finsler case is seen in \cite{Shiohama0}. The proof is omitted here.
						\par\medskip
						In the pioneering works~\cite{GM} and \cite{CG}, the authors have proved that a Busemann function\index{Busemann function} on a complete and non-compact Riemannian manifold $(M,g)$ is strongly convex if its sectional curvature is positive (see~\cite{GM}) and convex if its sectional curvature is non-negative (see~\cite{CG}). In particular, every super Busemann function is a convex exhaustion if its sectional curvature is non-negative. If the minimum set of a super Busemann function has non-empty boundary, then the negative of the distance function on the minimum set to the boundary is convex, and hence attains its minimum. Thus by iterating this, a totally convex compact totally geodesic submanifold without boundary, called a soul of $M$, is found in the minimum set. The well-known Cheeger--Gromoll structure Theorem\index{Cheeger--Gromoll structure Theorem}\index{Theorem!Cheeger--Gromoll} (see \cite{CG}) states that a complete non-compact Riemannian manifold is homeomorphic to the normal bundle over the soul in $M$. If the sectional curvature is positive, the soul is a point, and hence $M$ is diffeomorphic to $\R^n$. The Sharafutdinov construction \cite{Sharafutdinov} of flow curves along the negative of the subgradient of a Busemann function gives a distance non-increasing correspondence between two such flow curves. This was employed by Perelmann \cite{Perelman} for the proof of the famous soul conjecture. \par
						A simple example is seen here.
						
						\begin{tikzpicture}[scale=1.4]
						\draw [thick](-3.5,0) to(3.5,0) node[right] {$x$};
						\draw [thick](0,-2.6) to (0,2.6)node[above] {$y$};
						\draw [thick][decoration={markings, mark=at position 0.6 with {\arrow{<}}},postaction={decorate}](.2,1) to(1.5,2.3) node[above] {$-\nabla\phi$};
						\draw [thick][decoration={markings, mark=at position 0.6 with {\arrow{<}}},postaction={decorate}](.2,-1) to (1.5,-2.3)node[below] {$-\nabla\phi$};
						\draw (-1,0) to(0,-1) ;	\draw (0,-1) to(1,0) ;	\draw (1,0) to (0,1) ;	\draw (0,1) to(-1,0) ;
						\draw (-2,0) to(0,-2) ;	\draw (0,-2) to(2,0) ;	\draw (2,0) to (0,2) ;	\draw (0,2) to(-2,0) ;
						
						\draw[dashed][decoration={markings, mark=at position 0.8 with {\arrow{>}}},postaction={decorate}] (-1.2,2) arc (180:205:2.2cm);
						\draw[dashed][decoration={markings, mark=at position 0.8 with {\arrow{>}}},postaction={decorate}] (-2.2,1.4) arc (200:250:2.2cm);
						\draw[dashed][decoration={markings, mark=at position 0.8 with {\arrow{>}}},postaction={decorate}] (1.2,-2) arc (270:330:1cm) node[above] {flow curve};
						\draw[dashed][decoration={markings, mark=at position 0.8 with {\arrow{>}}},postaction={decorate}] (1.2,2) arc (200:260:1cm) node[right] {flow curve};
						\node at (2.2,-.2){$(b,0)$};\node at (1.2,-.2){$(a,0)$}; \node at (-1.3,-0.2){$(-a,0)$}; \node at (-2.3,-0.2){$(-b,0)$};
						\node at (-1,2.2) {$\phi^{-1}({b})$};\node at (-2,1.6) {$\phi^{-1}({a})$}; 
						\node[below=3.8cm]  {	Figure \ 2};
						\end{tikzpicture}
						\begin{Ex}
							Let $\phi(x,y):=|x|+|y|$ for $(x,y)\in\R^2$ be a convex function. It is clear that the distance function is monotone non-increasing along two flow curves of $-{\rm grad}(\phi)$.
						\end{Ex}			
						We do not know however, if anologues of the above results stated in \cite{GM}, \cite{CG} and in Example 5.2 are valid on 
						complete Finsler manifolds with positive (or non-negative) flag curvature (Flag curvature in Finsler geometry is an analogue of sectional curvature in Riemannian geometry, for details see \cite{BCS}).\par\medskip
						
						\subsection{Convex functions}\par
						A function $\phi:(M,F)\to\R$ is said to be {\it convex} if along every geodesic $\g:[a,b]\to (M,F)$, the restriction $\phi\circ\g:[a,b]\to\R$ is convex:  
						\begin{equation}\label{eq:convex}
						\phi\circ\g((1-\lambda)a+\lambda b)\le(1-\lambda)\phi\circ\g(a)+\lambda\phi\circ\g(b),\quad 0\le\lambda\le 1   
						\end{equation}   
						If the inequality in the above (\ref{eq:convex}) is strict for all $\g$ and for all $\lambda\in(0,1)$, $\phi$ is called {\it strictly convex}, and {\it strongly convex} if the second order difference quotient, namely $\{\phi o\g(h)-\phi o\g(-h) -2\phi o\g(0)\}/h^2$ is positive for all $\g$ and all $\lambda\in(0,1)$. In the special case where equality in (\ref{eq:convex}) holds for every $\g$ and for every $\lambda\in[0,1]$, the function is called an {\it affine function}. If a non-trivial convex function $\phi$ is constant on an open set, then it assumes its minimum on the open set and the number of components of a level set $\phi^{-1}(\{a\})$, $a>\inf_M\phi$, is equal to that of the boundary components of the minimum set of $\phi$. A convex function $\phi$ is said to be {\it locally non-constant} if it is not constant on any non-empty open set of $M$. From now on, {\it we always assume that our convex functions are locally non-constant}.\par
						The slope inequality of a one variable convex function is elementary, useful and employed throughout this section. Let $f:(\a,\beta)\to\mathbb R$ be a convex function. Let 
						$\a<a<b<c<d<\beta$ and $A:=(a,f(a))$, $B:=(b,f(b))$, $C:=(c,f(c))$ and $D:=(d,f(d))$ be points on the graph of $f$ (see Figure \ 3). The slope inequality is expressed as
						\begin{equation}\label{eq:slope}
						{\rm slope}(AB)\le{\rm slope}(AC)\le{\rm slope}(BC)\le{\rm slope}(BD)\le{\rm slope}(CD).
						\end{equation}
						
						\begin{tikzpicture} [scale=2.0]
						
						\draw[thick] (-1.8,3.24) parabola bend (0,0) (1.8,3.24) ;
						\draw (-1.2,1.44) -- (-0.5,0.25) node[left] {$B$};
						\draw (0.7,0.49) -- (-1.2,1.44) node[left] {$A$};	
						\draw (-0.5,0.25) -- (0.7,0.49)node[right] {$C$};
						\draw (0.7,.49) -- (1.6,2.56)node[right]{$D$};
						\draw (-0.5,0.25) -- (1.6,2.56);
						\node[below=0.4cm]  {	Figure \ 3};
						\end{tikzpicture}
						\par 
						We observe from (\ref{eq:slope}) that the right and left derivatives $f'_+(t)$ and $ f'_-(t)$ of $f$ exist and $f'_+(t)\ge f'_-(t)$ for all $t\in(\a,\beta)$, and  the equality holds if and only if $f$ is differentiable at $t$.
						\par\medskip
						
						The topology of Riemannian manifolds admitting locally nonconstant convex functions has been investigated in \cite{GS1} and \cite{GS2}. The topology of complete Alexandrov surfaces admitting locally nonconstant convex functions has been studied in \cite{Mashiko1} and \cite{Mashiko3}. The classification of Busemann G-surfaces admitting convex functions has been obtained in \cite{Innami1}. 
						The isometry groups of complete Riemannian manifolds admitting strictly convex functions have been discussed in \cite{Yamaguchi}. There are several extensions of convex functions, such as peakless\index{peakless function} functions introduced by Busemann~\cite{Busemann}, uniformly locally convex filtrations~\cite{Yamaguchi-convexfiltration}, etc. The splitting theorem for Alexandrov surfaces admitting affine functions has been established in \cite{Mashiko2}. Also in \cite{Mashiko1}, the condition for compact Alexandrov surfaces to admit locally non-constant convex functions has been studied. A detailed discussion of  convex sets on Riemannian manifolds of non-negative curvature is carried out in \cite{BZ}. It is emphasized that the notion of convexity makes sense, irrespective of whether the distance function is symmetric or not. Hence the extended notion of convex functions will be discussed on Finsler manifolds.\par\medskip
						We define the ends of a non-compact manifold $M$ and proper maps on $M$, which are useful for the investigation of topology of manifolds admitting convex functions.\par
						\begin{Def}\par
							{\rm
								Let $M$ be a noncompact manifold and $C_1\subset C_2\subset M$ be compact sets. Then, each component of $M\setminus C_2$ is contained in a unique component of $M\setminus C_1$.  An end of 
								$X$ is by definition an element of the inverse limit system $\{{\rm components\; of}\, M\setminus C\, ;\, C \,{\rm is\; compact}\}$ directed by the inclusion relation. 
							}
						\end{Def}\par
						For example, if $M$ is a compact surface from which $k$ points are removed, then $M$ has $k$ ends.
						$\R^1$ has two ends. For $n>1$, $\R^n$ has one end, for it is homemorphic to $S^n$ with one point removed. A cylinder $\bS^{n-1}\times\R$ has two ends, for it is homeomorphic to $S^n$ from which two points are removed.  \par
						\begin{Def}
							{\rm
								A map $f:M\to\R$ is said to be {\it proper} if  $f^{-1}(K)$ is compact for all compact set $K\subset \R$ and $f$ is said to be an {\it exhaustion} of $M$  if $f^{-1}(-\infty,a]$ is compact for all $a\in f(M)$.
							}
						\end{Def}
						\par\medskip
						
						We shall review the topology of geodesically complete Finsler manifolds admitting locally non-constant convex functions. The following Propositions are basic and important facts and have already been established in Riemannian geometry, (see \cite{GS1} and \cite{GS2}). \par
						Without assuming the continuity of a convex function on $(M,F)$, we have the following Proposition. For its proof see (\cite{GS1}, \cite{SS}):
						\begin{Prop}\label{Prop:Lipschitz}
							{\rm   
								Any convex function on $(M,F)$ is locally Lipschitz.
							}
						\end{Prop}
						
						\subsection{Riemannian and Finslerian results on convex functions}\par
						The assumption for a convex function to be locally non-constant, as was introduced, is necessary. For, we can construct on every noncompact manifold a complete Riemannian metric and a non-trivial smooth convex function whose minimum set contains a non empty open set. Therefore the existence of such a non-trivial convex function gives no restriction on the topology of a manifold. We first discuss the level sets of a locally nonconstant convex function. 
						\begin{Prop}\label{Prop:hypersurface}{\rm [compare Proposition 2.3 in \cite{GS1}]}
							{\rm
								Let $\phi:(M,F)\to\R$ be a convex function and $a>\inf_M\phi$. Then the $a$-level set $\phi^{-1}(\{a\})$ is a topological submanifold of dimension $n-1$.
							}
						\end{Prop}
						\medskip\noindent
						\begin{Prop}
							{\rm
								Let $C\subset(M,F)$ be a closed locally convex set. Then there exists a totally geodesic submanifold $W$ of $M$ such that $W\subset C$ and its closure is $C$. 
							}
						\end{Prop}
						\begin{proof}
							{\rm
								Since $C$ is locally convex, every point $x\in C$ admits a convex set $B(x,r)\cap C$ for some $r\in(0,\de(x))$. If $y$ is a point in $B(x,r)\cap C$, then $\g_{xy}:[0,d(x,y)]\to (M,F)$ and $\g_{xy}:[0,d(y,x)]\to (M,F)$ are contained entirely in $B(x,r)\cap C$. Clearly, its interior is a totally geodesic submanifold of $(M,F)$ of dimension at least one, contained in $B(x,r)\cap C$. Thus $x$ is contained in a $k(x)$-dimensional totally geodesic submanifold which is contained entirely in $B(x,r)\cap C$ such that $k(x)$ is maximal dimension of all such totally geodesic submanifolds in $B(x,r)\cap C$. Setting $k=\max_{x\in C}\,k(x)$, we have a $k$-dimensional totally geodesic submanifold, say, $W(x)$ of $(M,F)$ contained in $B(x,r)\cap C$. Suppose $W(x)\cap B(x,r)\subsetneqq C\cap B(x,r)$. We then find a point $z\in B(x,r)\cap(C\setminus W(x))$. Clearly, we have
								\[ d(z, \overline{W(x)} )\ge 0\quad\text{and $\dot\g_{xz}(0)$ is transversal to the tangent space  $T_xW(x)$ at $x$.}    \]
								Thus we find a small open set $\Om\subset W(x)$ and a family of minimizing geodesics emanating from points on $\Om$ and ending at $z$, whose initial vectors are transversal to $TW(x)$. We thus get a cone consisting of minimizing geodesics 
								\[    \g_{yz}:[0,d(y,z)]\to B(x,r)\cap C,\quad y\in\Om,    \]
								which is contained entirely in $C$ and forms a totally geodesic submanifold of dimension $k+1$, a contradiction.    \par
								Let $W=\cup_{x\in C}W(x)\subset C$. Again the transversality argument with $W(x)$ implies that $W$ is a smooth totally geodesic submanifold of maximal dimension in $C$. \par
								We finally prove that the closure $ \overline{W}$ of $W$ coincides with $C$. To prove this, suppose that there exists a point $z\in C\setminus\overline{W}$. We then find a point $y\in \overline{W}$ such that $d(z,y)=d(z,\overline{C}))>0.$ Let $T_y\overline{W}\subset T_yM$ be the linear subspace obtained as the limit $\lim_{j\to\infty}\,T_{y_j}W$; $y_j\in W$, $\lim_{j\to\infty}y_j=y$. If $\dot\g_{zy}(d(z,y))$ is transversal to $T_y\overline{W}$. Then the above argument shows the existence of a $(k+1)$-dimensional totally geodesic smooth submanifold in $C$, which is a contradiction. Therefore, we have $\dot\g_{zy}(d(z,y))\in T_y\overline{W}$, and hence $\g_{zy}(0,d(z,y))\subset W$. This proves $C=\overline{W}$.
							}
						\end{proof}
						\par\medskip
						\subsection{Level set configurations}\par
						An elementary observation based on the slope inequality (\ref{eq:slope}) gives the following simple fact on a locally non-constant convex function $\phi:(M,F)\to\R$ (and $\phi:(M,g)\to\R$). \par\medskip
						
						If there exists a compact level $\phi^{-1}(\{a\})\subset (M,g)$, then so are all the other levels.  If $\phi^{-1}(\{a\})\subset(M,F)$ is compact, then so are $\phi^{-1}(\{b\})$ for all $b\ge a$.
						
						The proof is sketched as follows: Suppose there is a non-compact level $\phi^{-1}(\{c\})\subset(M,F)$. There is a sequence of minimizing geodesics emanating from an arbitrary  fixed point $x\in \phi^{-1}(\{a\})$ and ending at points $y_j\in \phi^{-1}(\{c\})$, $j=1,2,\dots$ where $\{y_j\}_{j=1,2,\dots}$ is a  sequence of points with $\lim_{j\to\infty}d(x,y_j)=\infty$.  We then choose a ray obtained as the limit of these minimizing geodesics, along which $\phi$ must be bounded above by $c$. This means that this ray is contained in $\phi^{-1}(\{a\})$, a contradiction.\par\medskip
						\begin{Lem}\label{Lem:levelsets}
							{\rm
								Let $\phi:(M,F)\to\R$ be a locally non-constant convex function. Let $\phi^{-1}(\{a\})$ be a compact level. Then, 
								$\phi^{-1}[a,b]$ for a fixed $b>a$ is homeomorphic to the product $\phi^{-1}(\{a\})\times[a,b]$.
							}
						\end{Lem}
						\begin{proof}
							{\rm
								The slope inequality \ref{eq:slope} plays an important role in this. We first choose two numbers $a_0\in(\inf_M\phi,a)$ and $a_{k+1}>b$ and let $\de:=\de(\phi^{-1}[a_0,a_k])$ be the convexity radius over $\phi^{-1}[a_0,a_k]$. Take a sequence of real numbers
								\[   a_{k+1}>a_k:=b>a_{k-1}>\cdots>a_1:=a>a_0,     \]
								such that for each integer $i=2,\cdots,k+1$ and for each point $x\in\phi^{-1}(\{a_i\})$, we find a unique foot $f(x)$ of $x$ on $\phi^{-1}(\{a_{i-2}\})$.   \par
								Let $x_k\in\phi^{-1}(\{a_k\})$ be an arbitrary point. We then find a unique point $x_{k+1}\in\phi^{-1}(\{a_{k+1}\})$ such that $x_k$ belongs to the interior of $T(x_{k+1},f(x_{k+1}))$, where we denote by $T(x_{k+1},f(x_{k+1}))$ the unique foot of $x_{k+1}$ on $\phi^{-1}(\{a_{k-1}\})$. \par
								The uniqueness of feet implies that there is a homeomorphism between $\phi^{-1}(\{a_k\})$ and $\phi^{-1}(\{a_{k-1}\})$ via the correspondence $x_k\mapsto x_{k-1}$. Thus we have a homeomorphism between $\phi^{-1}[a_{k-1},a_k]$ and $\phi^{-1}(\{a_k\})\times[a_{k-1},a_k]$ through the feet.\par
								By iteration, we have for an arbitrary fixed point $x_k\in\phi^{-1}(\{a_k\})$, a sequence of points and a minimizing geodesics
								\[ \{x_{k+1},x_k,\cdots,x_1\}\ {\rm and} \ \{T(x_{k+1},f(x_{k+1})), T(x_k,f(x_k)),\cdots, T(x_2,f(x_2))\}   \]
								which satisfies the conditions
								\[      d(x_i,x_{i-1})>d(x_i,f(x_{i+1})),\quad i=k,\cdots,2.   \]
								
								\begin{tikzpicture}
								\draw (-4,0) to[bend left] (4,0) node[right] {$a_{k-2}$};
								\draw (-4,1) to[bend left] (4,1)node[right] {$a_{k-1}$};
								\draw (-4,2) to[bend left] (4,2)node[right] {$a_k=b$};
								\draw (-4,3) to[bend left] (4,3)node[right] {$a_{k+1}$};
								
								\draw[thick][decoration={markings, mark=at position 0.4 with {\arrow{>}}},postaction={decorate}] (.7,4.12) arc (120:155:2.2cm);
								\draw[dashed][decoration={markings, mark=at position 0.8 with {\arrow{>}}},postaction={decorate}] (.7,4.12) arc (120:182:2.2cm)node[below] {$f(x_{k+1})$};
								\draw[thick][decoration={markings, mark=at position 0.4 with {\arrow{>}}},postaction={decorate}] (-.2,3.16) arc (120:158:2.2cm);
								\draw[dashed][decoration={markings, mark=at position 0.8 with {\arrow{>}}},postaction={decorate}] (-.2,3.16) arc (120:186:2.2cm)node[below] {$f(x_k)$};
								\draw[thick][decoration={markings, mark=at position 0.4 with {\arrow{>}}},postaction={decorate}] (-1.15,2.05) arc (120:161:2.2cm);
								\draw[dashed][decoration={markings, mark=at position 0.8 with {\arrow{>}}},postaction={decorate}] (-1.15,2.05) arc (120:186:2.2cm);
								\node at (.87,4.3){$x_{k+1}$}; \node at (-.34,3.4){$x_k$};\node at (-1.7,2.3){$x_{k-1}$};\node at (-2.5,1.05){$x_{k-2}$};
								\draw[dashed][decoration={markings, mark=at position 0.9 with {\arrow{>}}},postaction={decorate}] (-.9,2.5)  arc (250:175:2.5cm)node[above] {$T(x_{k+1},..., x_1)$};

								\node[below=0.8cm]  {	Figure \ 4};
								\end{tikzpicture} \par
								The right derivative of $\phi oT(x_{k+1},\dots, x_1)$ is monotone increasing (this is evident from the figure). \par
								\begin{tikzpicture}
								\draw (-4,0) to[bend left] (4,0) node[right] {$a_0$};
								\draw (-4,1) to[bend left] (4,1)node[right] {$a_1=a$};
								\draw (-4,2) to[bend left] (4,2)node[right] {$a_2$};
								\draw (-4,3) to[bend left] (4,3)node[right] {$a_3$};
								\draw(-4,4) to[bend left] (4,4)node[right] {$a_4$};
								\draw[thick][decoration={markings, mark=at position 0.4 with {\arrow{>}}},postaction={decorate}] (0,5.18) arc (50:17:2.2cm);			
								\draw[dashed][decoration={markings, mark=at position 0.8 with {\arrow{>}}},postaction={decorate}] (0,5.18) arc (50:-10:2.2cm)node[below] {$f(x_4)$};
								\draw[thick][decoration={markings, mark=at position 0.4 with {\arrow{>}}},postaction={decorate}] (.7,4.15) arc (120:155:2.2cm);
								\draw[dashed][decoration={markings, mark=at position 0.8 with {\arrow{>}}},postaction={decorate}] (.7,4.15) arc (120:182:2.2cm)node[below] {$f(x_3)$};
								\draw[thick][decoration={markings, mark=at position 0.4 with {\arrow{>}}},postaction={decorate}] (-.2,3.16) arc (120:158:2.2cm);
								\draw[dashed][decoration={markings, mark=at position 0.8 with {\arrow{>}}},postaction={decorate}] (-.2,3.16) arc (120:186:2.2cm)node[below] {$f(x_2)$};
								\node at (0,5.5){$x_4$};\node at (.87,4.3){$x_3$}; \node at (-.34,3.4){$x_2$};\node at (-1.3,2.3){$x_1$};
								\draw[dashed][decoration={markings, mark=at position 0.9 with {\arrow{>}}},postaction={decorate}] (-.9,2.5)  arc(250:170:2.5cm)node[above] {$T(x_{k+1},..., x_1)$};
								\node[below=0.8cm]  {	Figure \ 5};
								\end{tikzpicture}

								The slope inequality then implies that the right and left derivatives of $\phi\circ T(x_{i+1},f(x_{i+1}))$ at $x_i$ are larger
								than those of $\phi\circ T(x_i,f(x_i))$ at $x_i$. Therefore, if $\phi$ is restricted to the union of broken geodesics 
								\begin{equation}\label{eq:brokengeodesics}    T(x_k,x_{k-1},\dots, x_1):=T(x_k,x_{k-1})\cup T(x_{k-1},x_{k-2})\cup\dots \cup T(x_2,x_1),   
								\end{equation}
								then, it is monotone and convex. Clearly, the right and left derivatives at every point of $T(x_k,x_{k-1},\dots, x_1)$ are bounded above by a negative number $\mu=\mu(a_0,a,\de)$;
								here $\mu$ is defined by
								\[     -\mu:=\frac{a-a_0}{\max\{d(x,\phi^{-1}(\{a_0\}))|x\in \phi^{-1}(\{a\})\}}.   \]
								This means that the length of $T(x_k,x_{k-1},\dots, x_1)$ is bounded above by $\frac{(b-a)}{\mu(a_0,a,\de)}<0$. This completes the proof.      
							}
						\end{proof}
						\medskip
						Let $\phi:(M,g)\to\R$ be a locally non-constant convex function on a complete Riemannian manifold. The Sharafutdinov construction of flow curves along $-{\rm grad}(\phi)$ implies that the diameter function $t\mapsto {\rm diam}(\phi^{-1}(\{t\}))$ is monotone non-decreasing. Here we set
						\begin{equation}   \label{eq:diam}{\rm diam}(\phi^{-1}(\{t\})):=\sup\,\{d(x,y)\,|\,x,y \in \phi^{-1}(\{t\})\subset (M,g)\}.     \end{equation}
						The monotone property of $t\mapsto{\rm diam}(\phi^{-1}(\{t\}))$ may be roughly explained as follows: \par
						Let $C\subset(M,g)$ be a closed convex set and $x,y\in M\setminus C$ be taken sufficiently close to $C$ such that there exist unique foot $f(x), f(y)$ of $x,y$ on $C$ respectively. We observe that $d(x,y)\ge d(f(x),f(y))$. This shows that the diameter function has everywhere non-negative derivative (See \cite{GS2}, \cite{GS1}). We therefore get that if $\phi$ admits a level that is compact, then so are all the others. \par
						However we do not know if the monotone property of the diameter function is valid for Finsler manifolds. Irrespective of whether the  diameter function is monotone or not,  we get from Lemma \ref{Lem:levelsets} the following
						\par\medskip
						\begin{Th}[see Theorem 1.1 in \cite{SS}]\label{Th:compactlevels}
							{\rm
								Let $(M,F)$ be a complete Finsler manifold and $\phi:(M,F)\to\R$ be a locally non-constant convex function whose level sets are all compact. Then we have the following:
								\begin{enumerate}
									\item If $\phi^{-1}(\{c\})$ is connected for some 
									$c>\inf_M\phi$, then there exists a homeomorphism $H:\phi^{-1}(\{c\})\times(\inf_M\phi,\infty)\to M$ such that\par
									(a) : $H(x,t)\in\phi^{-1}(\{t\})$ for all $(x,t)\in\phi^{-1}(\{c\})\times(\inf_M\phi,\infty)$.\par
									(b) : If $a,b\in\phi(M)$, $a<b$, we then have $H(x,[a,b])=T(x_k,x_{k-1},\dots,x_1)$ as defined in (\ref{eq:brokengeodesics}).
									\item If $\phi$ attains its infimum, say $m:=\inf_M\phi$, then $M$ is homeomorphic to the normal bundle over $\phi^{-1}(\{m\})$ in $M$.
									\item If there is a disconnected level, then $\phi$ attains its minimum $m=\inf_M\phi$, and $\phi^{-1}(\{m\})$ is a compact totally geodesic smooth hypersurface with trivial normal bundle. Moreover, $M$ is homeomorphic to $\phi^{-1}(\{m\})\times\R$.
								\end{enumerate}
							}
						\end{Th}  
						\medskip
						\begin{Rem}
							{\rm
								Without the assumption of the existence of a compact level of $\phi:(M,g)\to\R$, all the above statements are still valid in the Riemannian case. However we do not yet know this in the Finsler case.
							}
						\end{Rem}

						\subsection{Properness of exponential maps}\par
						The slope inequality of convex functions along geodesics leads us to the properness of the exponential maps on manifolds with convex functions. Clearly the exponential map $\exp_x:T_xM\to M$ at each point on a complete and simply connected Riemannian manifold $M$ gives a diffeomorphism, and hence it is proper. The proof is sketched as follows:\par
						First of all, let $(M,g)$ be a complete non-compact Riemannian manifold of positive sectional curvature. Then a super Busemann function $F_x:(M,g)\to\R$ at a point $x\in M$ is a strictly convex exhaustion. Under this condition Gromoll and Meyer \cite{GM} proved that the exponential map $\exp_x:T_xM\to M$ is proper. In fact, suppose to the  contrary that there is a compact set $K\subset M$ such that 
						$\exp_x^{-1}(K)$ is non-compact. Then there exists a divergent sequence $\{u_j\}_{j=1,2,\dots}\subset T_xM$ of vectors with $\lim_{j\to\infty} \Vert u_j\Vert=\infty$ such that $\exp_x\,u_j\in K$ for all $j=1,2,\dots$. Thus we find a geodesic $\g:[0,\infty)\to(M,g)$ such that $\phi\circ\g$ is bounded above, and hence it is constant. This is impossible, for $\phi$ is strictly convex. Hence the exponential map $\exp_x:T_xM\to M$ is proper. The properness of exponential map on Finsler manifold has recently been extended as follows:\par\medskip
						\begin{Th}[see \cite{ST}]\label{Th:proper}
							{\rm
								If $(M,F)$ is a geodesically complete non-compact Finsler manifold, and if $\phi:(M,F)\to\R$ is a strictly convex exhaustion function, then the exponential map at each point of $(M,F)$ is proper.
							}
						\end{Th}
						Notice that the exhaustion property in Theorem \ref{Th:proper} is needed for the conclusion to hold. For instance, let $\mathcal F\subset \mathbb R^3$ be a surface of revolution with profile curve $y=e^x$, $x\in\mathbb R$.
						Then the exponential map is not proper at any point of $\mathcal F$ .
						\par\medskip
						
						\subsection{Number of ends}\par
						The number of ends of complete Riemannian manifolds admitting locally non-constant convex functions is estimated by using the slope inequality (\ref{eq:slope}), Lemma \ref{Lem:levelsets} and Theorem \ref{Th:compactlevels}. 
						\par\medskip
						
						\begin{Th}[Ends of $(M,g)$, \cite{GS1}]\label{Th:ends1} 
							{\rm
								Let $(M,g)$ be a connected geodesically complete Riemannian manifold admitting a locally nonconstant convex function $\phi$. 
								\begin{enumerate}
									\item If $\phi$ has a noncompact level, then $M$ has one end. 
									\item If $\phi$ assumes its minimum and if it admits a compact level, then $M$ has one end.
									\item If $\phi$ has a disconnected compact level, then $M$ has  two ends.
									\item If $\phi$ has a compact level and if its infimum is not attained,
									then $M$ has two ends.
								\end{enumerate}
							}
						\end{Th}
						
						However the ends of geodesically complete Finsler manifolds admitting locally non-constant convex functions have not been fully understood yet. We do not know any example of a convex function $\phi:(M,F)\to\R$ with compact and non-compact levels simultaneously. \par\medskip
						\subsection{Isometry groups}\par
						Let $(H,g)$ be a Hadamard manifold, namely $H$ is a complete and simply connected Riemannian manifold of non-positive sectional curvature. Then the distance function $d(x,.)$ from a  fixed point $x\in H$ is convex, with a unique minimum point at $x$. 
						A  well-known classical theorem by Cartan states that if $G$ is a compact subgroup of the isometry group $I(H)$ of $H$, then it has a common fixed point. In fact, if $x\in H$ is a fixed point, then the $G$-orbit $G(x)$ of $x$ is compact, and hence there exists a unique smallest ball $B(y,r)$ with $G(x)\subset \bar B(y,r)$. Clearly $B(y,r)$ is invariant under the action of $G$, and hence the center $y$ is fixed under the actions of $G$. \par
						We finally discuss how the existence of a convex function on $(M,g)$ and $(M,F)$ influences the group of isometries on them. The splitting theorem for Riemannian manifolds admitting affine functions has been discussed in \cite{Innami2}. It is proved in \cite{GS3} that if $(M,g)$ is a complete Riemannian manifold with non-compact isometry group and if $(M,g)$ admits a convex function without minimum whose levels are all compact, then $(M,g)$ is isometric to the Riemannian product $N\times\mathbb R$, where $N$ is a compact smooth manifold. In \cite{CG} Cheeger and Gromoll  constructed the compact totally convex filtration obtained by a super Busemann function on a complete non-compact Riemannian manifold of non-negative sectional
						curvature. They proved:
						\begin{Th}\label{Th:CG}\par
							{\rm
								A complete Riemannian manifold $(M,g)$ of non-negative sectional curvature splits off isometrically as the product:
								\[      M= \overline{M}\times \mathbb R^k,       \]
								where the isometry group $I(\overline{M})$ of $\overline{M}$ is compact and $I(M)=I(\overline{M})\times I(\mathbb R^k)$.
							}
						\end{Th}
						Without assuming that the sectional curvature is non-negative, there are some results on the relation between the isometry groups and convex functions defined on $(M,g)$ and on $(M,F)$ respectively.\par
						\begin{Th}[see \cite{Yamaguchi}]\label{Th:isometries1}\par
							{\rm
								Let $(M,g)$ be a complete Riemannian manifold admitting a strictly convex function $\psi:(M,g)\to \mathbb R$. We then have
								\begin{enumerate}
									\item If $\psi$ admits a minimum, then every compact subgroup $G$ of the isometry group of $(M,g)$ has a common fixed point.
									\item  If $\psi$ has a compact level and if it has no minimum, then the group of isometries of $(M,g)$ is compact.
								\end{enumerate}
							}
						\end{Th}
						\begin{proof} {\rm For the proof of (1), we denote by $\mu$ the Haar measure on $G$, normalized by $\int_G d\mu=1 $. We define a function $\Psi:(M,g)\to \mathbb R$ by $\Psi(x):=\int_G\psi(gx) d\mu(g)$, $x\in M$. Clearly, $\Psi$ is strictly convex. Since $G$ is compact, $\Psi$ attains its minimum. The strict convexity of $\Psi$ means that the minimum set of $\Psi$ consists of a single point. It follows from the construction of $\Psi$ that the minimum set is a common fixed point of $G$.\par
								The strictly increasing property of the diameter function defined in (\ref{eq:diam}) plays an important role for the proof of (2). This fact can intuitively be understood as follows: \par
								Choose numbers $\inf_M\Psi<a<b$ such that $d(x,\psi^{-1}(\{a\}))$, for every $x\in \psi^{-1}(\{b\})$,  is less than the convexity radius on the compact set $\psi^ {-1}[a,b]$. If $x,y\in \psi^{-1}(\{b\})$ are sufficiently close to each other and if $f(x)$ and $f(y)$ are feet on $\psi^{-1}(\inf_M\Psi,a]$, we then have $d(x,y)>d(f(x),f(y))$, (see Figure\ 6).
								\begin{center}	\begin{tikzpicture} [scale=1.5]
									\draw[thick] (4,1) parabola bend (0,2) (-4,0) node[left] {$\psi^{-1}(\{a\})$} ;
									
									\draw[thick] [<-] (-2,1.5)to (-3,4) node[left]{$x$};
									\draw [thick][<-] (2,1.75)to (3,4)node[right]{$y$};
									\draw [thick](-2,1.5) to (2,1.75)node[below]{$f(y)$};
									\node at(-2,1.2) {$f(x)$};
									\node[above=0.3cm] {$\psi^{-1}(\inf_M \psi,a]$};
									\node[below=0.8cm]  {	Figure \ 6};
									\end{tikzpicture}
								\end{center}
								Roughly speaking, this is because of the angle property: $\measuredangle	(x,f(x),f(y))>\frac{\pi}{2}$ and 	  $	\measuredangle (y,f(y),f(x))>\frac{\pi}{2}$. Here  $\measuredangle(x,f(x),f(y))$ is the angle between two vectors at $f(x)$ tangent to minimizing geodesics joining $f(x)$ to $x$ and $f(x)$ to $f(y)$. This infinitesimal version of the above observation will give the Sharafutdinov construction of the distance non-increasing strong deformation retract.\par 
								For the proof of (2), we argue by deriving a contradiction. Suppose that the isometry group $\textbf{G}$ of $M$ is non-compact. Then the orbit $\textbf{G}(x)$ of an arbitrary  point $x\in (M,g)$ forms an unbounded set. We know from Theorems 5.1 and 5.3 that $M$ is homeomorphic to $\psi ^{-1}(\{a\})\times \mathbb R$, where $a:=\psi(x)$. Since the diameter function ${\rm diam}_{\psi}(t)$ of $\psi$ is strictly increasing, every isometry $g_1$ of $(M,g)$ fixes each end of $M$. We may chose an element $g_1\in \textbf{G}$ so as to satisfy: $g_1\circ\psi^{-1}(\{a\})$ is contained in $\psi[b,c]$, where $b-a$ is sufficiently large. Thus the diameter function of $\psi$ satisfies ${\rm diam}_{\psi}(b)>{\rm diam}_{\psi}(a)$. We then choose a proper curve $\alpha:(\inf_M\psi, \infty)\to (M,g)$ such that $\psi \circ \alpha$ is strictly increasing and $\alpha[b,c]$ does not meet $g_1\circ\psi^{-1}(\{a\})$ (see Figure\ 7).  It obviously follows that $g_1\circ\alpha:(\inf_M\psi, \infty)\to (M,g) $ does not pass through any point of $\psi^{-1}(\{a\})$ and join the two ends of M, a contradiction. } \end{proof}

								\begin{tikzpicture} [scale=0.7]
								\draw (.5,-10.5) arc (165:120:20cm);  \draw (-.5,-10.5) arc (15:60:20cm);
								\draw[dashed](-.8,-9.5) to[bend left] (.8,-9.5)node[right]{$\psi^{-1}(\{a\})$};
								\draw (-.8,-9.5) to[bend right] (.8,-9.5);
								\draw[dashed](-2.6,-5.5) to[bend left] (2.6,-5.5)node[right]{$\psi^{-1}(\{b\})$};
								\draw (-2.6,-5.5) to[bend right] (2.6,-5.5);
								\draw[dashed](-6.7,-.6) to[bend left] (6.7,-.6)node[right]{$\psi^{-1}(\{c\})$};
								\draw (-6.7,-.6) to[bend right] (6.7,-.6);
								\path (-4,-3.4) edge[ out=40, in=-40, looseness=0.8, loop, distance=2cm, ->]	node[above=3pt] {} (-4,-3.4);
								\draw (0,-9) arc (184:162:32cm) [decoration={markings, mark=at position 0.825 with {\arrow{>}}},postaction={decorate}]node[above]{$\alpha$};
								\node at (.8,-5.8){$\alpha(b)$};\node at (1,-1.8){$\alpha(c)$};
								\node at (-4.7,-3.5){$g_1(x)$};\node at (-2.3,-2.8){$g_1\circ\psi^{-1}(\{a\})$}; \node at (-1,-9.5) {$x$};
								\node at (2,-11){Figure \ 7};				
								\end{tikzpicture}
								
								We know very little about the isometry groups of complete Finsler manifolds admitting convex functions. Proof of the following result can  be found  in \cite{ST}.\par
								\begin{Th}[see \cite{ST}]\par
									{\rm
										Let $\psi:(M,F)\to\R$ be a strictly convex exhaustion function. Then every compact subgroup of the group of isometries on $(M,F)$ has a common fixed point.
									}
								\end{Th}
								
										\printindex

\end{document}